%% file: paper.tex
\documentclass[11pt]{elsarticle}

\usepackage{amsmath,amssymb,amsfonts,textcomp,amsthm,xifthen,psfrag,graphicx,color,MnSymbol}
\usepackage[T1]{fontenc}
\usepackage{hyperref}
\usepackage{placeins}
\usepackage{xargs}
\usepackage{enumitem}

\date{}

\input{header}

\def\new#1{#1}
\title{Minimum residual discretization of a semilinear elliptic problem\tnoteref{grants}}
\tnotetext[grants]{Supported by ANID-Chile through FONDECYT projects 1230013, 11240731,
                  and through the Austrian Science Fund (FWF) under grants P33216, PAT3446525.
                  NH acknowledges support through a guest professorship at the Vienna Center
                  for PDEs funded by TU Wien.}
\author[1]{Carlos Garc\'\i a Vera}
\ead{carlos.garcia@ulagos.cl}
\affiliation[1]{organization={Departamento de Ciencias Exactas, Universidad de Los Lagos},
               city={Osorno}, country={Chile}}

\author[2]{Norbert Heuer\corref{cor}}
\ead{nheuer@uc.cl}
\affiliation[2]{organization={Facultad de Matem\'aticas,
                              Pontificia Universidad Cat\'olica de Chile},
               city={Santiago}, country={Chile}}
\cortext[cor]{Corresponding author}

\author[3]{Dirk Praetorius}
\ead{dirk.praetorius@asc.tuwien.ac.at}
\affiliation[3]{organization={Institut f\"ur Analysis und Scientific Computing,
                              Technische Universit\"at Wien},
                city={Vienna}, country={Austria}}

\begin{document}

\begin{abstract}
We propose a least-squares penalization as a means to extend the
discontinuous Petrov--Galerkin (DPG) method with optimal test functions
to a class of semilinear elliptic problems.
The nonlinear contributions are replaced with independent
unknowns so that standard DPG techniques apply to the \new{resulting} linear problem
with non-trivial kernel. The nonlinear relations are added as
least-squares constraints. Assuming solvability of the semilinear problem
and an Aubin--Nitsche-type approximation property for the primal variable, we prove
a C\'ea estimate for the approximation error in canonical norms.
\new{We also give conditions ensuring unique solvability and, under additional regularity
and stability assumptions, derive the required nonlinear Aubin--Nitsche estimate.}
Numerical results with uniform and adaptively refined meshes
illustrate the performance of the scheme, also in comparison with the Gauss--Newton approach.
\end{abstract}


\maketitle
\section{Introduction and model problem}

In this paper, we propose an extension of the DPG technique with optimal test functions
to solve semilinear elliptic problems. Our method falls into the class of
finite-element minimum residual discretizations.

In the last 15 years, the discontinuous Petrov--Galerkin (DPG) method with optimal test functions
has been established as a framework for robust approximation schemes of
various classes of problems, e.g., starting with transport and
Laplace-type problems
\cite{DemkowiczG_10_CDP,NiemiCC_13_ASD,DemkowiczG_11_ADM,BramwellDGQ_12_Lhp},
also in non-divergence form \cite{Fuehrer_21_UFL},
extensions to wave problems and electromagnetics
\cite{ZitelliMDGPC_11_CDP,DemkowiczGMZ_12_WEA,GopalakrishnanMO_14_DDE,DemkowiczGNS_17_SDM,CarstensenDG_16_BSF},
singularly perturbed problems \cite{DemkowiczH_13_RDM,HeuerK_17_RDM,FuehrerH_24_RDT},
parabolic problems \cite{FuehrerHS_17_TSD,FuehrerHK_21_ABE,DieningS_22_STD},
fourth-order problems and the deformation of thin structures
\cite{NiemiBD_11_DPG,CaloCN_14_ADP,FuehrerHN_19_UFK,FuehrerGH_21_LFD,FuehrerHN_22_DMS,FuehrerHN_23_DMR},
boundary integral equations \cite{HeuerP_14_UFH,HeuerK_17_DPG},
problems set in Banach spaces \cite{MugaVdZ_20_DLP},
and control problems \cite{FuehrerF_24_DML}.
For a more complete and detailed discussion, we refer to \cite{DemkowiczG_25_DPG}.

The underlying idea of this Petrov--Galerkin framework is to generate discrete stability
by the selection of so-called optimal test functions, made efficient by the use
of product (or ``broken'') test spaces. The efficiency of this approach is inherently
related \new{to} the linearity of the underlying problem.
The development of efficient DPG schemes for nonlinear problems is still a major challenge.
This paper is a contribution to \new{addressing} this hurdle.

The existing literature on DPG analysis for nonlinear problems is scarce
and, as far as we are aware, consists of the following four approaches.
\new{As in least-squares methods} \cite[Chap.~8]{BochevG_09_LSF},
Demkowicz and co-authors propose a Gauss--Newton scheme which solves
sequences of linearized problems, in this case by the DPG method.
We refer to \cite{ChanDM_14_DMS,RobertsDM_15_DPG} for fluid-flow problems,
and the detailed discussion in \cite[Sec.~10]{DemkowiczG_25_DPG}.
Carstensen and co-authors propose and analyze a lowest-order
nonlinear minimum residual setting \cite{CarstensenBHW_18_NDP} (see also \cite{ZhanD_25_NED}),
whereas Cantin and Heuer \cite{CantinH_18_DFS} develop an extended DPG setting
for strongly monotone operators.
More recently, the group of Demkowicz has studied the steepest descent method in
combination with the DPG framework and its relation with the Gauss--Newton approach,
see \cite{Li_24_thesis,ZhangD_26_NED} and the exposition in \cite[Sec.~10]{DemkowiczG_25_DPG}.

In addition to these approaches there are other cases where problems are formally
nonlinear and require specific constructions:
linear PDEs with nonlinear constraint \cite{BuiThanhG_14_PCO};
contact problems \cite{FuehrerHS_18_DMS};
eigenvalue problems \cite{BertrandBS_23_DPG};
and problems in Banach spaces \cite{MugaVdZ_20_DLP} where the construction of optimal
test functions becomes nonlinear, see also \cite{LiD_21_LDM,LiD_22_LDM}.

In this paper, we propose a DPG setting for semilinear elliptic problems.
Our problems are motivated by the Richards equation \cite{Richards_31_CCL}.
It is nonlinear of degenerate parabolic type and used to describe groundwater flow
in unsaturated media. Some early results for its numerical analysis can be found
in \cite{RaduPK_04_OCE,PopRK_04_MFE,SchneidKR_04_PEE}, and
\new{the problem has attracted considerable recent interest.}
A time-stepping approach for this
equation, e.g. as in \cite{FuehrerHS_17_TSD,FuehrerHK_21_ABE},
leads to a sequence of semilinear problems of second order with coercive principal part
and, in several applications, Lipschitz continuous advective and reactive terms.
We represent both nonlinear terms by independent variables, thus giving rise to a linear
problem, and add the nonlinear relations as constraints.

This idea of introducing additional variables is similar to the approach in
\cite{CantinH_18_DFS}. There, strongly monotone diffusion is considered
and the nonlinear constraint is added in variational form in such a way
that the resulting system \new{preserves} strong monotonicity. This requires
a sufficiently large stabilization parameter.
The Lipschitz and strong monotonicity properties of the formulation
in \cite{CantinH_18_DFS} guarantee its well-posedness and the quasi-optimal convergence
of conforming discretizations.
In contrast, the problems we consider here are semilinear, and our main result only
assumes Lipschitz continuity of the nonlinear functions, together with a superconvergence
property further discussed below. Rather than adding the nonlinear constraints
in variational form, here we impose them as least-squares terms, similarly
to the least-squares DPG-BEM coupling in \cite[\S 2.5.1]{FuehrerHK_17_CDB}.
We also specify conditions \new{for existence and uniqueness of a solution.
These conditions} effectively imply strong monotonicity,
and the technique from \cite{CantinH_18_DFS} becomes applicable.

We approximate the linearly extended problem by the DPG method,
and complement it with finite-element least squares for the constraints.
Globally, the method is a finite-element minimum residual method that minimizes
the linear residual in a discretized dual test space and nonlinear residuals in $L_2$.
A major advantage of this approach is that all residuals consist of local terms
associated with elements. They can be calculated and serve in a straightforward
manner as indicators for adaptive mesh refinements,
cf.~\cite{DemkowiczGN_12_CDP,CarstensenDG_14_PEC}.
We see the treatment of nonlinearities as additional least-squares
constraints (where possible) as the natural extension of the DPG framework in the sense
of a posteriori error estimation. The DPG advantages in this respect extend
to nonlinear problems, in particular the built-in error control by local residuals and
the observation that even coarse initial meshes lead to adequate mesh refinements,
cf.~\cite{DemkowiczGN_12_CDP,KeithVD_19_GOA}.
A second advantage is that the nonlinear terms are confined to
simple $L_2$ residuals and that the more expensive steps of the DPG part
(assembling of matrices, inversion of the Riesz operator)
have to be performed only once in the solution of the nonlinear systems by the Newton method.
This comes at the cost of an additional variable for every nonlinear term.
In the numerical section, we compare the efficiency of our
DPG-least squares method plus Newton solver (DPG-LS/Newton)
with the Gauss--Newton (DPG) scheme which is among the standard procedures for
nonlinear problems, not just for DPG discretizations.
Our results illustrate similar costs of both methods for a given mesh
(even though DPG-LS/Newton uses two more field variables) and superior
error control of the DPG-LS/Newton scheme for certain nonlinearities.

We note that there are recent analyses \cite{BertrandBRS_LSF,BringmannP_26_GCA}
of least squares finite elements for nonlinear problems.
The approach in \cite{BertrandBRS_LSF}
assumes small data and is not applicable in our case.
On the other hand, \cite{BringmannP_26_GCA} makes strong assumptions on the
quasilinear nonlinearity to show global linear convergence of the Zarantanello linearization.
We assume the existence of a solution to the semilinear problem,
and show a C\'ea error estimate in canonical norms
for approximations sufficiently close to a solution.
A critical ingredient is an Aubin--Nitsche type error estimate
for the primal variable. In \cite{BoumaGH_14_CRD}, Gopalakrishnan and co-authors
verify this \new{property} for linear problems, see also the superconvergence properties proved
by F\"uhrer \cite{Fuehrer_18_SDM,Fuehrer_19_SDM}.
\new{We show the Aubin--Nitsche estimate for the semilinear case under additional
assumptions that, in particular, require some smoothness of the nonlinear functions.}
We stress the fact that the Aubin--Nitsche \new{property} relies on a
primal DPG setting as in \cite{BoumaGH_14_CRD,DemkowiczG_13_PDM,Niemi_23_APP} rather than
on the predominant ultraweak formulations, in particular considered in
\cite{Fuehrer_18_SDM,Fuehrer_19_SDM}.
In this context, we note that primal formulations, at least from the analysis
point of view, seem to be better suited for parabolic problems than ultraweak formulations,
cf. the results in \cite{FuehrerHS_17_TSD} (ultraweak) versus \cite{FuehrerHK_21_ABE} (primal).
In that sense, our primal setting should be the sensible choice to deal with the
parabolic Richards equation. This is left to future research.

Our model problem reads as follows.
Given a bounded, connected Lipschitz domain $\Omega\subset\R^d$ ($d=2,3$)
and $f\in L_2(\Omega)$, find $u\in H^1_0(\Omega)$ with
\begin{align} \label{pde}
   -\div(\kappa\grad u + \rho(u)\beta) + \gamma(u) = f\quad\text{in}\ \Omega.
\end{align}
Here, for simplicity, $\beta\in\R^d$ is constant,
$\kappa\in L_\infty(\Omega;\R^{d\times d})$ is symmetric, uniformly positive definite,
and $\rho,\gamma$ are Lipschitz continuous functions.
(This can be reduced to the Lipschitz continuity in a neighborhood
of the range of a solution $u$, cf. Remark~\ref{rem_Lip} below.)

Of course, the right-hand side function $f$ can be assumed to be zero by a change of
the nonlinearity $\gamma$, \new{then also depending on $x\in\Omega$}.
\new{For convenience in constructing manufactured} solutions,
we consider a right-hand side function.
\new{For ease of presentation, our analysis is restricted to homogeneous Dirichlet data.
In the numerical section, non-homogeneous data are imposed by fixing the boundary
degrees of freedom, so that the Newton updates satisfy homogeneous boundary conditions.
We do not analyze the corresponding extension here.}
A thorough convergence analysis of the employed Newton approach as well as the required
interplay of the stopping criterion with mesh adaptivity is left to future research;
see, e.g., \cite{HeidW_20_AIL,HeidPW_21_ECO,Heid_24_SNA,MiraciPS_27_UCL,BringmannBP_NMA}
in the context of standard finite elements.

The remainder of this paper is as follows.
In the next section, we introduce required operators and spaces, present a minimum residual
formulation of the model problem, prove their equivalence (Theorem~\ref{thm_min}),
formulate conditions under which the existence of solutions and their uniqueness
can be shown (Theorem~\ref{thm_ex}),
and represent the kernel of the linear extended problem (Proposition~\ref{prop_ker}).
In Section~\ref{sec_disc}, we present and analyze the discrete scheme,
including a C\'ea error estimate (Theorem~\ref{thm_Cea}).
\new{In Section~\ref{sec_AN}, we derive a nonlinear Aubin--Nitsche estimate under additional
regularity and stability assumptions and use it to establish local quasi-optimality.}
A numerical implementation is discussed in Section~\ref{sec_num}.
This includes the representation of the nonlinear systems (\S\ref{sec_Newton})
which are solved by the Newton method, the discussion of adaptivity (\S\ref{sec_adap}),
and the presentation of numerical experiments for smooth and singular solutions
(\S\ref{sec_exp}).

Throughout the paper, the notation $A\lesssim B$ indicates that
$A\le c B$ for a positive constant $c$ independent of functions involved and any mesh
that may be present. Relation $A\gtrsim B$ means $B\lesssim A$.

\section{Minimum residual formulation} \label{sec_min}

This section develops the minimum residual formulation of the model problem, and establishes
an equivalence \new{of the two} problems. We also formulate conditions for the
existence and uniqueness of solutions, and identify the kernel of the linear part of
the formulation.
We start by introducing some notation, spaces, and a trace operator which are used throughout
the remainder.

For relatively open $\omega\subset\overline\Omega$,
we use standard Sobolev spaces $L_2(\omega)$, $H^1(\omega)$, $H^1_0(\omega)$,
the gradient and divergence operators $\grad$, $\div$,
and denote by $\vdual{\cdot}{\cdot}_\omega$ and $\|\cdot\|_\omega$ the $L_2(\omega)$
duality and norm, respectively (generically for scalar and vector-valued functions).
We drop the index $\omega$ when $\omega=\Omega$.
We also need the norms $\|\cdot\|_{\div}:=(\|\cdot\|^2+\|\div\cdot\|^2)^{1/2}$
and $\|\cdot\|_{1,\omega}:=(\|\cdot\|_\omega^2+\|\grad\cdot\|_\omega^2)^{1/2}$.
We introduce a regular triangulation $\mesh$ of $\Omega$ and denote by $\vdual{\cdot}{\cdot}_\mesh$
and $\|\cdot\|_\mesh$ the corresponding duality and norm in the product space $L_2(\mesh)$.
We also need the product space $H^1(\mesh):=\Pi_{\el\in\mesh} H^1(\el)$. For an element
$\el\in\mesh$, $\dual{\cdot}{\cdot}_{\partial\el}$ stands for the $L_2$ duality on the boundary
$\partial\el$ of $\el$,
\new{and we also use $\dual{\cdot}{\cdot}_F$ for
the $L_2(F)$ duality on subsets (faces) $F\subset\partial\el$.}

Integration by parts for $\mesh$-piecewise smooth functions gives rise to traces on
the mesh skeleton $\cS$. More formally, we define the normal trace operator
\begin{align*}
   \trn:\;&\left\{\begin{array}{cll}
               H(\div,\Omega) &\rightarrow& H^1(\mesh)^*,\\
               \tau &\mapsto& \dual{\trn(\tau)}{v}_\cS
               := \vdual{v}{\div\tau} + \vdual{\grad v}{\tau}_\mesh
            \end{array}\right.
\end{align*}
with image
\[
   H^{-1/2}(\cS) := \trn(H(\div,\Omega)).
\]
We furnish this space with the canonical trace norm
\[
   \|\hsigma\|_{-1/2,\cS} := \inf\{\|\sigma\|_{\div};\; \trn(\sigma)=\hsigma\},
   \quad \hsigma\in H^{-1/2}(\cS).
\]
Furthermore, we define the space
\[
   U:=H^1_0(\Omega)\times H^{-1/2}(\cS)\times L_2(\Omega)\times L_2(\Omega)
\]
with norm
\[
   \|\bw\|_U := \Bigl(\|\grad w\|^2 + \|\htau\|_{-1/2,\cS}^2 + \|s\|^2 + \|t\|^2\Bigr)^{1/2},
   \quad \bw=(w,\htau,s,t)\in U.
\]
We introduce $q:=\rho(u)$ and $r:=\gamma(u)$. Then, problem \eqref{pde} reads
\[
   -\div(\kappa\grad u+q\beta)+r=f\quad\text{in}\ \Omega,
\]
written variationally with test space $V:=H^1(\mesh)$ as
\begin{align} \label{prob_bL}
   b(u,\hsigma,q,r;v):=
   \vdual{\kappa\grad u+q\beta}{\grad v}_\mesh - \dual{\hsigma}{v}_\cS + \vdual{r}{v}
   = L(v) := \vdual{f}{v}
   \quad\forall v\in V
\end{align}
where $\hsigma:=\trn(\kappa\grad u+q\beta)\in H^{-1/2}(\cS)$.
We split the problem into a linear part with operator
\[
   B:\;U\to V^*,\quad
   \bu\mapsto B\bu:\ V\ni v\mapsto b(\bu;v)
\]
and the nonlinear relations $q=\rho(u)$, $r=\gamma(u)$.
This transforms problem \eqref{pde} into the minimum residual formulation
\new{
\begin{align} \label{min}
   \bu\in\argmin_{\bw\in U} J(\bw)\quad \text{with}\quad
   J(\bw):=
   \frac 12\Bigl(\|B\bw-L\|_{V^*}^2 + \|\rho(w)-s\|^2 + \|\gamma(w)-t\|^2\Bigr)
\end{align}
for $\bw=(w,\htau,s,t)$.}
Norm $\|\cdot\|_{V^*}$ is the standard norm of the topological dual $V^*$ of $V$.
We note that minimizing the residual $\|B\cdot-L\|_{V^*}$ is equivalent to solving
\begin{equation} \label{prob_bLTheta}
   b(\bu,\Theta \bw)=L(\Theta\bw)\quad\forall \bw\in U
\end{equation}
for $\bu\in U$. Here, $\Theta:\;U\to V$ is the trial-to-test operator defined by
\[
   \ip{\Theta\bw}{v}_V = b(\bw,v)\quad\forall v\in V
\]
with inner product $\ip{v}{w}_V:=\vdual{v}{w}+\vdual{\grad v}{\grad w}_\mesh$
for $v,w\in V$. Of course, formulations \eqref{prob_bLTheta} and \eqref{prob_bL} are
equivalent. For details we refer to \cite{DemkowiczG_25_DPG}.
A discretization of \eqref{prob_bLTheta} amounts to a DPG method.
Therefore, a discretization of the minimum residual formulation \eqref{min}
can be interpreted as a combination of the DPG and least-squares methods.

Model problem \eqref{pde} and formulation \eqref{min} are equivalent in the following sense.

\begin{theorem} \label{thm_min}
Suppose that \eqref{pde} has a solution.
Then, any solution $u\in H^1_0(\Omega)$ of \eqref{pde} gives rise to a solution
$\bu=(u,\hsigma,q,r)$ of \eqref{min} with
$\hsigma=\trn(\kappa\grad u+q\beta)$, $q=\rho(u)$, and $r=\gamma(u)$.
On the other hand, any solution $\bu=(u,\hsigma,q,r)$ of \eqref{min} provides
a solution $u\in H^1_0(\Omega)$ of \eqref{pde} and satisfies
$\hsigma=\trn(\kappa\grad u +\rho(u)\beta)$.
\end{theorem}

\begin{proof}
For a solution $u$ of \eqref{pde},
$\bu:=(u,\new{\trn(\kappa\grad u+\rho(u)\beta)},\rho(u),\gamma(u))\in U$
leads to a vanishing residual in \eqref{min} by construction. It therefore is a minimizer.
\new{Conversely, let $\bu=(u,\hsigma,q,r)\in U$ be a minimizer of $J$.
Since \eqref{pde} has a solution, the minimum value of $J$ is zero. Hence,}
$B\bu=L$, $\rho(u)=q$, $\gamma(u)=r$. We conclude that
$b(u,\hsigma,\rho(u),\gamma(u);v)=L(v)$ for any $v\in V$ and this implies
that $u\in H^1_0(\Omega)$ satisfies \eqref{pde} and
$\hsigma=\trn(\kappa\grad u +\rho(u)\beta)$.
\end{proof}

It is clear that additional conditions guarantee the existence of solutions
to semilinear problems. Such conditions are formulated in the following result.
A proof is given in \ref{sec_app}.

\begin{theorem} \label{thm_ex}
\new{In addition} to the assumptions made on the domain and coefficients $\beta, \kappa$,
we assume (with some redundancy) the following:
\begin{itemize}
\item[(A1)] $\exists\alpha>0$: $y^\top\kappa(x) y\ge \alpha|y|^2$ $\forall y\in\R^d$
            for a.e. $x\in\Omega$;
\item[(A2)] $\rho,\gamma:\;\R\to\R$ are Lipschitz continuous with respective
            Lipschitz constants $L_\rho, L_\gamma$, and $\rho$ is bounded;
\item[(A3)] $\gamma$ is continuously differentiable with $\gamma'\ge l_\gamma>0$
            on $\R$.
\end{itemize}
Then, given $f\in H_0^1(\Omega)^*$, problem \eqref{pde} has a weak solution.
If, additionally,
\begin{itemize}
\item[(A4)]
\(\displaystyle
   L_\rho^2 |\beta|^2<4 \alpha l_\gamma
\)
\end{itemize}
holds then problem \eqref{pde} gives rise to a Lipschitz continuous, strongly monotone
operator, and \eqref{pde} is uniquely solvable.
\end{theorem}

We need to identify the kernel of operator $B$. To this end, let us introduce the operator
$\cE:\;L_2(\Omega)\times L_2(\Omega)\to U$ defined by
\begin{align} \label{E}
   \cE(q,r):=(u,\hsigma,q,r)\quad \text{where}\quad u\in H^1_0(\Omega):\
   -\div\kappa\grad u=\beta\cdot\grad q-r,\quad
   \hsigma:=\trn(\kappa\grad u+q\beta).
\end{align}

\begin{prop} \label{prop_ker}
Operator $\cE:\;L_2(\Omega)\times L_2(\Omega)\to U$ is bounded with image
\(
   \cN(B) = \cE(L_2(\Omega)\times L_2(\Omega)),
\)
the kernel of $B$. Furthermore, $B:\;U/\cN(B)\to V^*$ is an isomorphism.
\end{prop}

\begin{proof}
Given $q,r\in L_2(\Omega)$, the right-hand side
$\beta\cdot\grad q-r$ in $\eqref{E}$ is an element of the dual space
$H^{-1}(\Omega):=H^1_0(\Omega)^*$ and therefore, the solution
$u\in H^1_0(\Omega)$ in \eqref{E} is well defined and bounded,
$\|\grad u\|\lesssim \|q\|+\|r\|$. By the continuity of $\trn$
and relation $\div(\kappa\grad u+q\beta)=r$, we also find that
\[
   \|\hsigma\|_{-1/2,\cS}^2\lesssim \|\kappa\grad u+q\beta\|_{\div}^2
   = \|\kappa\grad u+q\beta\|^2 + \|r\|^2 \lesssim \|q\|^2+\|r\|^2.
\]
This proves the boundedness of $\cE$.

It is clear that $B:\;H^1_0(\Omega)\times H^{-1/2}(\cS)\times\{0\}\times\{0\}\to V^*$
is an isomorphism, see, e.g., \cite{DemkowiczG_13_PDM}.
This shows that
\[
   \|B\bu\|_{V^*} = \|B(\bu-\cE(q,r))\|_{V^*} \gtrsim \|\bu-\cE(q,r)\|_U
   \quad\forall\bu=(u,\hsigma,q,r)\in U
\]
so that $\cE(q,r)=\bu$ if $B\bu=0$, implying $\cN(B)\subset\cE(L_2(\Omega)\times L_2(\Omega))$.
The inclusion $\cE(L_2(\Omega)\times L_2(\Omega))\subset \cN(B)$ holds by construction
and we conclude the claimed representation of $\cN(B)$.
By the closed range theorem, $B:\;U/\cN(B)\to V^*$ is an isomorphism.
\end{proof}

\section{Discretization} \label{sec_disc}

To discretize minimization problem \eqref{min}, we consider finite-dimensional subspaces
$V_h\subset V$ and $U_h\subset U$.
Subspace $V_h$ induces a discrete operator norm denoted as $\|\cdot\|_{V_h^*}$.
Throughout, we make use of the canonical \new{restriction of $V^*$ onto $V_h$}, i.e.,
functionals $L,B\bw\in V^*$ (for $\bw\in U$) are also considered to be elements of
$V_h^*$ with the same notation.

We note that the existence of a Fortin operator $\opF:\;V\to V_h$ with
\begin{align} \label{Fortin}
   \exists \CF>0\ \forall v\in V,\;\bu_h\in U_h:\quad
  \|\opF v\|_V\le \CF \|v\|_V,\quad
   b(\bu_h,v-\opF v) = 0
\end{align}
implies the discrete stability
\begin{align} \label{stab_disc}
   \|B\bu_h\|_{V_h^*} \le \|B\bu_h\|_{V^*} \le \CF \|B\bu_h\|_{V_h^*}\quad \forall \bu_h\in U_h,
\end{align}
see \cite[Proof of Theorem~2.1]{GopalakrishnanQ_14_APD}.

Using the discrete spaces $U_h$ and $V_h$, the minimum residual discretization of \eqref{pde} reads
\new{
\begin{align} \label{min_disc}
   \bu_h\in\argmin_{\bw_h\in U_h} J_h(\bw_h)\quad\text{with}\quad
   J_h(\bw_h) :=
   \frac 12 \Bigl(\|B\bw_h-L\|_{V_h^*}^2 + \|\rho(w_h)-s_h\|^2 + \|\gamma(w_h)-t_h\|^2\Bigr)
\end{align}
}
where $\bw_h=(w_h,\htau_h,s_h,t_h)$.
\new{Note that the functional $J_h$ is different from $J$ only by replacing the residual
norm in $V^*$ by $V_h^*$.}
As discussed before, this is a combination
of a DPG method for the linear part of the problem with finite-element least squares
for the nonlinear constraints.

In order to provide an error estimate, we need to make a closeness assumption.
For certain linear elliptic problems with solution $u\in H^1_0(\Omega)$
and finite element approximation $u_h$ in a discrete subspace of $H^1_0(\Omega)$,
the Aubin--Nitsche technique shows that $\|u-u_h\|\lesssim h^\alpha \|\grad(u-u_h)\|$
for an $\alpha>0$ depending on regularity properties and polynomial degrees,
and $h$ is a measure of the underlying mesh, e.g., the maximum of the elements' diameters.
Particularly for the primal DPG setting of linear problems, this property
is proved in \cite{BoumaGH_14_CRD}.
\new{In Section~\ref{sec_AN} we derive a nonlinear Aubin--Nitsche estimate
that is sufficient to close the quasi-optimality argument.}

\begin{theorem} \label{thm_Cea}
Assume that \eqref{pde} has a solution $u\in H^1_0(\Omega)$, and that
the finite-dimensional subspace $V_h\subset V$ allows for a Fortin operator
$\opF$, i.e., satisfying \eqref{Fortin}.
Let $\bu_h=(u_h,\hsigma_h,q_h,r_h)\in U_h$ be a solution of $\eqref{min_disc}$
such that $u_h$ is close to $u$ in the sense that
$\|u-u_h\|\le c(h)\new{\|\bu-\bu_h\|_U}$ with $0<c(h)\to 0$ as $h\to 0$.
Then, for sufficiently small $h$, $\bu_h$ satisfies the quasi-optimal error estimate
\begin{align*}
   \|\bu-\bu_h\|_U
   &\lesssim
   \|\bu-\bw_h\|_U
   \quad\forall \bw_h=(w_h,\htau_h,s_h,t_h)\in U_h.
\end{align*}
\end{theorem}

\begin{proof}
As mentioned, the existence of a Fortin operator implies the stability \eqref{stab_disc}.
Therefore, Proposition~\ref{prop_ker} and \eqref{stab_disc} show that any
$\bw_h=(w_h,\htau_h,s_h,t_h)\in U_h$ satisfies
\begin{align} \label{stab_disc2}
   \|\bw_h\|_U
   &\le \|\bw_h-\cE(s_h,t_h)\|_U + \|\cE(s_h,t_h)\|_U \nonumber\\
   &\lesssim \|B\bw_h\|_{V^*} + \|\cE(s_h,t_h)\|_U
   \lesssim  \|B\bw_h\|_{V_h^*} + \|s_h\| + \|t_h\|.
\end{align}
Adding and subtracting $\bw_h=(w_h,\htau_h,s_h,t_h)\in U_h$,
and using estimate \eqref{stab_disc2}, we deduce that
\begin{align} \label{pf1}
   \|\bu-\bu_h\|_U
   &\le \|\bu-\bw_h\|_U + \|\bu_h-\bw_h\|_U
   \nonumber\\
   &\lesssim \|\bu-\bw_h\|_U + \|B(\bu_h-\bw_h)\|_{V_h^*} + \|q_h-s_h\| + \|r_h-t_h\|
   \nonumber\\
   &\lesssim \|\bu-\bw_h\|_U + \|B\bu_h-L\|_{V_h^*} + \|B\bw_h-L\|_{V_h^*}
           + \|q_h-s_h\| + \|r_h-t_h\|.
\end{align}
We add and subtract $\rho(u)$, $\rho(u_h)$, and $\rho(w_h)$ to bound
\begin{align} \label{pf2}
   \|q_h-s_h\|
   \le  \|\rho(u)-\rho(u_h)\| + \|\rho(u_h)-q_h\| + \|\rho(u)-\rho(w_h)\| + \|\rho(w_h)-s_h\|
\end{align}
and, in the same way,
\begin{align} \label{pf3}
   \|r_h-t_h\|
   \le  \|\gamma(u)-\gamma(u_h)\| + \|\gamma(u_h)-r_h\|
   + \|\gamma(u)-\gamma(w_h)\| + \|\gamma(w_h)-t_h\|.
\end{align}
Minimization property \eqref{min_disc} implies the bound
\[
   \|B\bu_h-L\|_{V_h^*}^2 + \|\rho(u_h)-q_h\|^2 + \|\gamma(u_h)-r_h\|^2
   \le
   \|B\bw_h-L\|_{V_h^*}^2 + \|\rho(w_h)-s_h\|^2 + \|\gamma(w_h)-t_h\|^2.
\]
This, combined with estimates \eqref{pf1}, \eqref{pf2}, \eqref{pf3}, leads to
\begin{align*}
   \|\bu-\bu_h\|_U
   &\lesssim
   \|\bu-\bw_h\|_U
   + \|B\bw_h-L\|_{V_h^*} + \|\rho(w_h)-s_h\| + \|\gamma(w_h)-t_h\| \\
   &+ \|\rho(u)-\rho(w_h)\| + \|\gamma(u)-\gamma(w_h)\|
   + \|\rho(u)-\rho(u_h)\| + \|\gamma(u)-\gamma(u_h)\|
\end{align*}
so that, by the Lipschitz continuity of $\rho,\gamma$,
implying their Lipschitz continuity in $L_2(\Omega)$,
\begin{align} \label{err_bound}
   \|\bu-\bu_h\|_U
   &\lesssim
   \|\bu-\bw_h\|_U
   + \|B\bw_h-L\|_{V_h^*} + \|\rho(w_h)-s_h\| + \|\gamma(w_h)-t_h\| + \|u-w_h\| + \|u-u_h\|.
\end{align}
All but the last term on the right-hand side are controlled by $\|\bu-\bw_h\|_U$.
In fact, since $B\bu=L$, the continuous \new{restriction of $V^*$ onto $V_h$} with unit norm
and the boundedness of $B:\;U\to V^*$ show
\[
   \|B\bw_h-L\|_{V_h^*} = \|B(\bu-\bw_h)\|_{V_h^*}
   \le \|B(\bu-\bw_h)\|_{V^*} \lesssim \|\bu-\bw_h\|_U.
\]
An application of the triangle inequality and the Lipschitz continuity of $\rho$ and $\gamma$
imply
\begin{align*}
   &\|\rho(w_h)-s_h\| + \|\gamma(w_h)-t_h\| + \|u-w_h\|\\
   &\quad
    \le \|\rho(u)-\rho(w_h)\| + \|\rho(u)-s_h\|
      + \|\gamma(u)-\gamma(w_h)\| + \|\gamma(u)-t_h\| + \|u-w_h\|\\
   &\quad
    \lesssim \|u-w_h\| + \|\rho(u)-s_h\| + \|\gamma(u)-t_h\|.
\end{align*}
Recall that $\rho(u)=q$ and $\gamma(u)=r$. The Poincar\'e--Friedrichs inequality therefore
leads to
\[
   \|u-w_h\| + \|\rho(u)-s_h\| + \|\gamma(u)-t_h\|
   \lesssim \|\grad(u-w_h)\| + \|q-s_h\| + \|r-t_h\| \le 3\|\bu-\bw_h\|_U.
\]
A combination of the latter estimates and \eqref{err_bound} shows that
\[
   \|\bu-\bu_h\|_U \lesssim \|\bu-\bw_h\|_U + \|u-u_h\|.
\]
By assumption, $h$ is small enough such that there is a \new{sufficiently small}
constant $C>0$ so that $\|u-u_h\|\le C\new{\|\bu-\bu_h\|_U}$
can be covered by the left-hand side. This finishes the proof.
\end{proof}

\begin{remark} \label{rem_Lip}
The assumption of Lipschitz continuity of $\rho$ and $\gamma$ can be reduced
to their Lipschitz continuity in a neighborhood of the range of a solution $u$.
Then, the proof of \new{Theorem~\ref{thm_Cea}} applies to discrete functions
$\bw_h=(w_h,\htau_h,s_h,t_h)\in U_h$ with $w_h$
sufficiently close to $u$.
It is unknown whether the local Lipschitz continuity can be further
relaxed. In fact, theoretical results for nonlinear problems often assume more regularity.
Our method consists in solving the Euler--Lagrange equations of
\new{the discrete problem \eqref{min_disc}}.
They require differentiability of $\rho$ and $\gamma$, and the Newton solver calls for
a second differentiation. To deal with non-differentiable functions one has to
leave the variational Euler-Lagrange setting and use other solvers.
\end{remark}

Let us conclude this section by considering specific discretization spaces.
\new{To this end, from now on we assume that $\Omega$ is a polygonal/polyhedral domain.}
Let $\cP^p(\omega)$ be the space of polynomials on $\omega\subset\R^d$
of degree less than or equal to $p$.
We consider a regular mesh $\mesh$ on $\Omega$ of shape-regular simplices and denote
the set of faces of $\mesh$ by $\cF$ so that $\cS=\cup_{\face\in\cF}\face$.
For an element $\el\in\mesh$, $\cF(\el)$ denotes its set of faces (edges if $d=2$).
We define
\begin{align*}
   \cP^p(\mesh) &:= \{v\in L_2(\Omega);\; v|_\el\in \cP^p(\el),\ \el\in\mesh\},\\
   \cP^p(\cS) &:= \{v\in L_2(\cS);\; v|_\face\in \cP^p(\face),\ \face\in\cF\}
\end{align*}
and, for $\el\in\mesh$ and $\face\in\cF(\el)$,
\begin{align*}
   \cP^p_\face(\el) &:= \{v\in\cP^{d+p}(\el);\; v|_{\face'}=0,\ {\face'}\in\cF(\el)\setminus\{\face\}\}
   &&\text{(face bubble functions)},\\
   \cP^p_\el(\el) &:= \{v\in\cP^{d+p+1}(\el);\; v|_{\face'}=0,\ \face'\in\cF(\el)\}
   &&\text{(element bubble functions)},\\
   V_p(\el) &:= \cP^0(\el)+\sum_{\face\in\cF(\el)} \cP^p_\face(\el)+\cP^p_\el(\el).
\end{align*}
For a non-negative integer $p$, we select the discrete spaces
\begin{align} \label{Uh}
   U_h &:= \cP^{p+1}(\mesh)\cap H^1_0(\Omega)\; \times\; \cP^{p}(\cS)
           \;\times\; \cP^p(\mesh) \;\times\; \cP^p(\mesh)
\end{align}
and
\begin{align} \label{Vh}
   V_h := \Pi_{\el\in\mesh} V_p(\el)\quad \text{if}\ p>0,\qquad
   V_h := \Pi_{\el\in\mesh} \bigl(\cP^1(\el)+\cP^0_\el(\el)\bigr)\quad \text{if}\ p=0.
\end{align}
The construction of $V_h$ is due to \cite[Section 3]{FuehrerH_24_RDT}. These test spaces
are of lower dimension than previously used spaces, cf.~\cite{GopalakrishnanQ_14_APD}.

\begin{cor} \label{cor_Fortin}
\new{For element-piecewise constant $\kappa$ and polygonal/polyhedral domain $\Omega$,}
the selection of discrete spaces \eqref{Uh}, \eqref{Vh} allows for a Fortin operator and, thus,
\new{the quasi-optimal error estimate in} Theorem~\ref{thm_Cea} applies
\new{provided the Aubin--Nitsche type assumption holds.}
\end{cor}

\begin{proof}
According to \cite[Theorems~3.2, 3.8]{FuehrerH_24_RDT}, there is a bounded operator
$\opF:\;V\to V_h$
such that any $v\in \new{V}$ satisfies \new{for $\el\in\mesh$}
\[
   \vdual{\eta}{v-\opF v}_\el=0\ \forall \eta\in\cP^p(\el),\quad
   \dual{\mu}{v-\opF v}_{\new{\face}}=0\ \forall\mu\in \new{\cP^{p}(\face),\ \face\in\cF(\el)}.
\]
For $\bu_h=(u_h,\hsigma_h,q_h,r_h)\in U_h$ we find that
\begin{align*}
   b(\bu_h,v-\opF v) &=
   \vdual{\kappa\grad u_h+q_h\beta}{\grad (v-\opF v)}_\mesh
   - \dual{\hsigma_h}{v-\opF v}_\cS + \vdual{r_h}{v-\opF v}\\
   &=
   \sum_{\el\in\mesh} \vdual{\kappa\grad u_h+q_h\beta}{\grad (v-\opF v)}_{\el},
\end{align*}
and integration by parts on every element shows that the latter term vanishes.
This proves \eqref{Fortin} as claimed.
\end{proof}

\section{\new{The Aubin--Nitsche property}} \label{sec_AN}

As previously discussed, the Aubin--Nitsche property for linear primal DPG formulations
is proved in \cite{BoumaGH_14_CRD}. In the case of finite elements,
this property is known to hold for certain quasilinear and semilinear problems,
see, e.g., \cite{Schultz_71_LEB,Rannacher_77_SAE,DouglasD_75_GMN}.
There, several assumptions are made on the nonlinearities, in particular
requiring uniquely solvable problems.
Here, we follow the integrated second-variation argument from
Hardering's dissertation \cite{Hardering_thesis} and the published
manifold-valued version \cite{Hardering_18_LDE}; see also the Euclidean semilinear
exposition \cite{Hardering_ANT}.

In this section we make use of fractional-order Sobolev spaces $H^t(\omega)$
for open $\omega\subset\Omega$, cf.~\cite{Adams}. For $t\ge0$,
$\|\cdot\|_{t,\omega}$ denotes the norm in $H^t(\omega)$, and as before we write
$\|\cdot\|_t$ when $\omega=\Omega$.  The same notation is used for
vector-valued spaces.  We use this abbreviated notation only for Sobolev spaces
based on $L_2$; for $L_p$-based spaces ($p\not=2$)  the space is written explicitly in the norm
subscript, for instance $\|\cdot\|_{L_r}$, $\|\cdot\|_{W^{1,r}}$, and
$\|\cdot\|_{L_\infty}$.  For intersections we use the sum norm.  For
$0\le t<1/2$ we set
$H^{-t}(\Omega):=H_0^t(\Omega)^*$ and use the corresponding dual norm
$\|\cdot\|_{-t}$; for $t=0$ this agrees with $L_2(\Omega)$ by the Riesz
identification.  In particular, for the regularity index $s\in(1/2,1]$
used below, $H^{-1+s}(\Omega)$ is equipped with the norm
$\|\cdot\|_{-1+s}$.

Throughout, we assume that $\Omega$ is a polygonal/polyhedral (Lipschitz) domain.
We make the following additional assumptions and give corresponding remarks and implications.

\begin{enumerate}[label=(B\arabic*)]
\item \label{ass_B1}
We assume that $\kappa\in\R^{d\times d}$ is constant, symmetric, and positive definite,
and that there exists $s\in(1/2,1]$
such that the principal Dirichlet operator satisfies the regularity shift
\begin{equation}
      \|v\|_{1+s}
      \lesssim \|G\|_{-1+s}
      \label{eq:principal-shift}
\end{equation}
for every $G\in H^{-1+s}(\Omega)$ and the solution $v\in H_0^1(\Omega)$ of
  \[
      -\div(\kappa\nabla v)=G.
  \]
For domains in $d=2$, the existence of such an $s$ follows from standard elliptic regularity.
In three dimensions, it holds, for instance, for convex polyhedral domains.
\item \label{ass_B2}
  We assume that $\rho,\gamma\in C^2$ on an interval containing the ranges of the primal
  components on the segment between the exact solution and the discrete
  stationary point under consideration, and their first and second derivatives
  are bounded there.
\item \label{ass_B3}
  With $a:=\rho'(u)$ and $c:=\gamma'(u)$ for the solution $u$ of \eqref{pde},
  the linearized form
  \begin{equation}
    \cL(v,w)
      :=(\kappa\nabla v+a v\beta,\nabla w)+(cv,w)
      \label{eq:linform}
  \end{equation}
  is coercive on $H_0^1(\Omega)$,
  \begin{equation}
    \exists\mu>0:\quad\cL(v,v)\ge \mu\|v\|_1^2\quad\forall v\in H^1_0(\Omega).
    \label{eq:lincoercive}
  \end{equation}
  A sufficient condition is
  \[
      c\ge c_0>0,
      \qquad
      \|a\|_{L_\infty}^2|\beta|^2<4\alpha c_0,
  \]
  where $\xi^\top\kappa\xi\ge\alpha|\xi|^2$.
\item \label{ass_B4}
  The mesh $\mesh$ is shape regular and simplicial and the trial space $U_h$
  is given as in \eqref{Uh}.  For $p=0$ the test space $\wilde V_h:=V_h$ is as in \eqref{Vh},
  while for $p>0$ we use the enrichment
  \[
       \wilde V_h
       :=\prod_{\el\in\mesh}\bigl(V_p(\el)+\cP^1(\el)\bigr).
  \]
  The Fortin operator constructed in \cite{FuehrerH_24_RDT} remains a Fortin
  operator for this enlarged space, since its range is contained in the original
  test space.  For the flux approximation we use the standard Raviart--Thomas
  interpolant $\Pi_h^{RT}$ of order $p$.  Since $s>1/2$, the required normal
  traces are well defined, and the usual approximation and commuting properties
  give, for the fields occurring below,
  \begin{equation}
    \|\sigma-\Pi_h^{RT}\sigma\|_{\div}
    \lesssim h^s\bigl(\|\sigma\|_s
            +\|\div\sigma\|_s\bigr),
    \label{eq:RT-approx}
  \end{equation}
  cf.~\cite{BoffiBF_13_MFE}.
\end{enumerate}

The replacement of $V_h$ by $\wilde V_h$ in the definition of the discrete functional
$J_h$ in \eqref{min_disc} gives a new functional $\wilde J_h$.
Our discretization then minimizes $\wilde J_h$, but for the analysis we also use its
evaluation on $U$. Therefore, to be precise, we define
\begin{align} \label{tildeJh}
   \wilde J_h(\bw) :=
   \frac 12 \Bigl(\|B\bw-L\|_{\wilde V_h^*}^2 + \|\rho(w)-s\|^2 + \|\gamma(w)-t\|^2\Bigr)
   \quad\text{for}\quad \bw=(w,\htau,s,t)\in U
\end{align}
and a discrete solution as $\bu_h\in\argmin_{U_h}\wilde J_h$.

Let $\widetilde P_h:V\to\widetilde V_h$ denote the $V$-orthogonal projection.
At the exact solution define
\begin{align*}
  \cH(\bz,\bw)
  :={}&\ip{\Theta\bz}{\Theta\bw}_V
       +(az-z_s,aw-w_s)+(cz-z_t,cw-w_t),
       \\
  \cH_h(\bz,\bw)
  :={}&\ip{\widetilde P_h\Theta\bz}{\widetilde P_h\Theta\bw}_V
       +(az-z_s,aw-w_s)+(cz-z_t,cw-w_t).
\end{align*}
Here $\bz=(z,\widehat\zeta,z_s,z_t)$ and analogously for
$\bw$.  Since all residuals vanish at the exact solution,
$\cH$ and $\cH_h$ are the Hessians of the continuous and enriched discrete
least-squares functionals, respectively.

First we show the coercivity of the linearized least-squares Hessian.

\begin{lemma}
Under \ref{ass_B2} and \ref{ass_B3},
\[
    \|\bz\|_{U}^2
    \lesssim \cH(\bz,\bz) \quad\forall\bz\in U.
\]
\end{lemma}

\begin{proof}
Put $r_s:=az-z_s$ and $r_t:=cz-z_t$.  For $v\in H_0^1(\Omega)$,
\[
  \cL(z,v)
   =\dual{B\bz}{v} +(r_s\beta,\nabla v)+(r_t,v).
\]
Coercivity of \eqref{eq:linform} therefore gives
\[
  \|z\|_1
  \lesssim \|B\bz\|_{V^*}+\|r_s\|+\|r_t\|.
\]
Hence $\|z_s\|+\|z_t\|$ is controlled by the same right-hand side.  The kernel
representation from Proposition~\ref{prop_ker} then yields
\[
  \|\bz\|_{U}
  \lesssim
  \|B\bz\|_{V^*}+\|r_s\|+\|r_t\|.
\]
Since $\|B\bz\|_{V^*}=\|\Theta\bz\|_V$, the assertion follows.
\end{proof}

Next we analyze the linearized adjoint problem.
For $g\in L_2(\Omega)$ define
\begin{equation}
  \bphi=(\phi,\widehat\phi,\phi_s,\phi_t)\in U:\quad
  \cH(\bz,\bphi)=(z,g) \quad\forall\bz\in U,
  \label{eq:adjoint-H}
\end{equation}
and set $\psi:=\Theta\bphi$.
Since $\cH$ is bounded and coercive, the Lax--Milgram lemma ensures that $\bphi$ is well defined.
Properties of $\bphi$ and $\psi$  are analyzed in the next two lemmas.

\begin{lemma}
\label{lem_adjointreg}
Under \ref{ass_B1}--\ref{ass_B3},
\begin{align}
  &\psi\in H_0^1(\Omega),
  \qquad
  a\phi-\phi_s=\beta\cdot\nabla\psi,
  \qquad
  c\phi-\phi_t=\psi,
  \label{eq:adjoint-identities}\\
  &-\div(\kappa\nabla\psi)
      +a\,\beta\cdot\nabla\psi+c\psi=g,
      \qquad \psi|_{\partial\Omega}=0,
  \label{eq:psi-equation}
\end{align}
and
\begin{equation}
  \|\psi\|_{1+s}+\|\phi\|_{1+s}
  +\|\phi_s\|_s+\|\phi_t\|_s
  +\|\phi\|_{L_\infty}
  \lesssim \|g\|.
  \label{eq:adjoint-reg}
\end{equation}
If
\[
  \sigma:=\kappa\nabla\phi+\phi_s\beta-\nabla\psi,
\]
then
\begin{equation}
   \widehat\phi=\operatorname{tr}_n \sigma,
   \qquad
   \div\sigma=\phi_t-\psi,
   \qquad
   \|\sigma\|_s+\|\div\sigma\|_s\lesssim\|g\|.
   \label{eq:flux-reg}
\end{equation}
\end{lemma}

\begin{proof}
Testing \eqref{eq:adjoint-H} successively with
$(0,\widehat\zeta,0,0)$, $(0,0,z_s,0)$, $(0,0,0,z_t)$, and $(z,0,0,0)$,
the first choice gives
\[
  \dual{\widehat\zeta}{\psi}_{\cS}=0
  \quad\forall \widehat\zeta\in H^{-1/2}(\cS),
\]
and therefore $\psi\in H_0^1(\Omega)$. The next two choices give
\[
  a\phi-\phi_s=\beta\cdot\nabla\psi,
  \qquad
  c\phi-\phi_t=\psi.
\]
With the last choice one obtains
\[
  (\kappa\nabla z,\nabla\psi)
  +(az,\beta\cdot\nabla\psi)+(cz,\psi)=(z,g),
\]
which is the weak form of \eqref{eq:psi-equation}.

The form in \eqref{eq:psi-equation} has the same diagonal as
\eqref{eq:linform}; hence \eqref{eq:lincoercive} gives
$\|\psi\|_1\lesssim\|g\|$.  Rewriting \eqref{eq:psi-equation} as
\[
  -\div(\kappa\nabla\psi)
   =g-a\beta\cdot\nabla\psi-c\psi
\]
shows that its right-hand side belongs to $L_2(\Omega)$ and hence to
$H^{-1+s}(\Omega)$.  The shift estimate \eqref{eq:principal-shift} therefore
implies
\begin{equation}
  \|\psi\|_{1+s}\lesssim\|g\|.
  \label{eq:psi-shift}
\end{equation}
We next recover $\phi$.  From $\psi=\Theta\bphi$,
\begin{equation}
 \ip{\psi}{v}_V
 = (\kappa\nabla\phi+\phi_s\beta,\nabla v)_{\mesh}
   -\dual{\widehat\phi}{v}_{\cS}
   +(\phi_t,v)
 \quad\forall v\in V.
 \label{eq:Theta-phi}
\end{equation}
For $v\in H_0^1(\Omega)$ the trace term vanishes.  Substituting
\eqref{eq:adjoint-identities} gives
\begin{equation}
  (\kappa\nabla\phi+a\phi\beta,\nabla v)+(c\phi,v)
   =2(\psi,v)
      +((I+\beta\otimes\beta)\nabla\psi,\nabla v).
  \label{eq:phi-equation}
\end{equation}
Coercivity first gives $\|\phi\|_1\lesssim\|g\|$.

We now use the same shift estimate once more.  First, the exact solution satisfies
$u\in H^{1+s}(\Omega)$.  Indeed, since $\beta$ is constant,
\[
  -\div(\kappa\nabla u)
   =f+\rho'(u)\beta\cdot\nabla u-\gamma(u)\in L_2(\Omega)
   \hookrightarrow H^{-1+s}(\Omega),
\]
so \eqref{eq:principal-shift} applies.  Since $s>1/2$ and $d\le3$, Sobolev embedding implies
\[
   H^s(\Omega)\hookrightarrow L_r(\Omega)
   \qquad\text{for some }r>d.
\]
Hence $\nabla u\in L_r(\Omega;\R^d)$, and the chain rule together with
the bounded
second derivatives in \ref{ass_B2} yields
\[
  a,c\in W^{1,r}(\Omega)\cap L_\infty(\Omega).
\]
Writing \eqref{eq:phi-equation} with principal part
$-\div(\kappa\nabla\phi)$ gives
\[
 -\div(\kappa\nabla\phi)
 =2\psi
  -\div\bigl((I+\beta\otimes\beta)\nabla\psi\bigr)
  +\div(a\phi\beta)-c\phi.
\]
The second term belongs to $H^{-1+s}(\Omega)$ by \eqref{eq:psi-shift}.
The remaining lower-order terms are even in $L_2(\Omega)$.  To see this for
$\div(a\phi\beta)$, set $m:=2r/(r-2)$.  Since $r>d$ and
$d\le3$, $H^1(\Omega)\hookrightarrow L_m(\Omega)$, and therefore
\[
  \nabla(a\phi)=a\nabla\phi+\phi\nabla a\in L_2(\Omega;\R^d),
\]
with
\[
  \|\nabla(a\phi)\|
  \le \|a\|_{L_\infty}\|\nabla\phi\|
      +\|\phi\|_{L_m}\|\nabla a\|_{L_r}
  \lesssim \|a\|_{W^{1,r}\cap L_\infty}\|\phi\|_1.
\]
Since $\beta$ is constant,
$\div(a\phi\beta)\in L_2(\Omega)$, and clearly
$c\phi\in L_2(\Omega)$.  Thus the complete right-hand side belongs to
$H^{-1+s}(\Omega)$, and \eqref{eq:principal-shift} gives
\[
  \|\phi\|_{1+s}\lesssim\|g\|.
\]
Since $s>1/2$ and $d\le3$, we also have
$H^{1+s}(\Omega)\hookrightarrow L_\infty(\Omega)$.  Finally,
\eqref{eq:adjoint-identities} and the product estimate above imply
\[
  \|\phi_s\|_s+\|\phi_t\|_s\lesssim\|g\|.
\]
Finally, with
\[
  \sigma=\kappa\nabla\phi+\phi_s\beta-\nabla\psi,
\]
subtracting $(\psi,v)+(\nabla\psi,\nabla v)_{\mesh}$ in
\eqref{eq:Theta-phi} yields
\[
  (\sigma,\nabla v)_{\mesh}
  -\dual{\widehat\phi}{v}_{\cS}
  +(\phi_t-\psi,v)=0
  \quad\forall v\in H^1(\mesh).
\]
The regularity already proved shows that
$\sigma\in H^s(\Omega;\R^d)\cap H(\div,\Omega)$.  Comparing
the last identity with the definition
of the normal trace gives
\[
  \div\sigma=\phi_t-\psi,
  \qquad
  \widehat\phi=\operatorname{tr}_n \sigma.
\]
Since $\phi_t\in H^s(\Omega)$ and $\psi\in H^{1+s}(\Omega)$, we also obtain
\[
  \|\sigma\|_s+\|\div\sigma\|_s\lesssim\|g\|,
\]
which completes the proof of \eqref{eq:adjoint-reg}--\eqref{eq:flux-reg}.
\end{proof}

\begin{lemma}
\label{lem_adjointapprox}
Under \ref{ass_B1}--\ref{ass_B4}, there exists
$\bphi_h\in U_h$ such that
\begin{align}
  \|\bphi-\bphi_h\|_{U}
   +\|(I-\widetilde P_h)\Theta\bphi_h\|_V
   &\lesssim h^s\|g\|,
   \label{eq:dual-approx}\\
  \|\bphi_h\|_{U}+\|\phi_h\|_{L_\infty}
   &\lesssim\|g\|.
   \label{eq:dual-stability}
\end{align}
\end{lemma}

\begin{proof}
Let $I_h^{p+1}\phi$ be the conforming Lagrange interpolant and let $Q_h^p$
be the elementwise $L_2$-projection.  With the Raviart--Thomas interpolant from
\ref{ass_B4}, set
\[
 \bphi_h
 :=\bigl(I_h^{p+1}\phi,
       \operatorname{tr}_n(\Pi_h^{RT}\sigma),
       Q_h^p\phi_s,Q_h^p\phi_t\bigr).
\]
For the trace component, continuity of the normal-trace map and
\eqref{eq:RT-approx} give
\[
 \|\widehat\phi-
      \operatorname{tr}_n(\Pi_h^{RT}\sigma)\|_{-1/2,\cS}
 \le \|\sigma-\Pi_h^{RT}\sigma\|_{\div}
 \lesssim h^s\|g\|.
\]
Together with fractional polynomial approximation for the remaining components,
Lemma~\ref{lem_adjointreg} therefore gives
\[
  \|\bphi-\bphi_h\|_{U}
  \lesssim h^s\|g\|;
\]
see, e.g., \cite{Ciarlet,BoffiBF_13_MFE}.  Since
$s>1/2$ and $d\le3$, the primal adjoint component is continuous.  Fixed-degree Lagrange
interpolation is locally $L_\infty$-stable on shape-regular meshes, which gives
\eqref{eq:dual-stability}.

For the second term in \eqref{eq:dual-approx}, use
$\psi=\Theta\bphi$:
\[
 \|(I-\widetilde P_h)\Theta\bphi_h\|_V
 \le
 \|\Theta(\bphi_h-\bphi)\|_V
 +\|(I-\widetilde P_h)\psi\|_V.
\]
The first term is $O(h^s\|g\|)$.  Since $\cP^1(\el)\subset\widetilde V_h|_\el$ and
$\psi\in H^{1+s}(\Omega)$, fractional piecewise affine approximation in the
broken $H^1$-norm shows that the second term is $O(h^s\|g\|)$.  This is the practical-DPG
consistency step; compare the primal-DPG Aubin--Nitsche analysis of
\cite{BoumaGH_14_CRD}.
\end{proof}

We are ready to state and prove an Aubin--Nitsche property of the primal variable
and consequently, quasi-optimal convergence.

\begin{theorem} \label{thm_AN}
Assume that $\Omega$ is a polygonal/polyhedral Lipschitz domain, and let
Assumptions \ref{ass_B1}--\ref{ass_B4} hold.
In particular, for $p>0$ the local test spaces $V_p(\el)$ are enriched by linear polynomials.
Furthermore, let $\bu_h\in U_h$ be a minimizer of the discrete functional $\wilde J_h$,
and select
\[
  \boldsymbol e:=\bu-\bu_h,
  \qquad e:=u-u_h,
  \qquad E:=\|\boldsymbol e\|_{U}.
\]
As long as the segment between $u$ and $u_h$ remains in the interval from
\ref{ass_B2},
\begin{equation}
  \|e\|\lesssim (h^s+E)E.
  \label{eq:nonlinear-AN}
\end{equation}
Consequently, there are $h_0,\delta>0$ such that, for $h\le h_0$ and
$E\le\delta$,
\begin{equation}
  \|\bu-\bu_h\|_{U}
  \lesssim
  \inf_{\bw_h\in U_h}
  \|\bu-\bw_h\|_{U}.
  \label{eq:local-cea}
\end{equation}
\end{theorem}

\begin{proof}
Take $g=e$ in \eqref{eq:adjoint-H} and let $\bphi_h$ be given by
Lemma~\ref{lem_adjointapprox}.  Then
\begin{equation*}
 \|e\|^2
 =\cH(\boldsymbol e,\bphi-\bphi_h)
 +(\cH-\cH_h)(\boldsymbol e,\bphi_h)
 +\cH_h(\boldsymbol e,\bphi_h).
\end{equation*}
The first term is bounded by $Ch^sE\|e\|$.  Orthogonality of
$\widetilde P_h$ gives the exact consistency identity
\[
 (\cH-\cH_h)(\bz,\bw)
 =\ip{(I-\widetilde P_h)\Theta\bz}{(I-\widetilde P_h)\Theta\bw}_V,
\]
so the second term is also bounded by $Ch^sE\|e\|$ using
Lemma~\ref{lem_adjointapprox}.

For the last term set
$\Gamma(\tau):=\bu_h+\tau(\bu-\bu_h)$.
The exact solution has zero residual and $\bu_h$ is stationary in discrete directions.
Hence, with enriched discrete functional $\wilde J_h$ from \eqref{tildeJh},
\begin{equation*}
 0
 =\int_0^1 \wilde J_h''(\Gamma(\tau))
       (\boldsymbol e,\bphi_h)\,\mathrm{d}\tau,
\end{equation*}
and therefore
\begin{equation}
 \cH_h(\boldsymbol e,\bphi_h)
 =\int_0^1
   \bigl[\wilde J_h''(\bu)-\wilde J_h''(\Gamma(\tau))\bigr]
   (\boldsymbol e,\bphi_h)\,\mathrm{d}\tau.
 \label{eq:hessian-integral}
\end{equation}
This is the nonlinear replacement of Galerkin orthogonality; see
\cite{Hardering_thesis,Hardering_18_LDE,Hardering_ANT}.  The DPG part is quadratic and
cancels in \eqref{eq:hessian-integral}.  For
$F_\rho(w,s)=\frac12\|\rho(w)-s\|^2$,
\[
 F_\rho''(w,s)[(z,z_s),(v,v_s)]
 = (\rho'(w)z-z_s,\rho'(w)v-v_s)
   +(\rho(w)-s,\rho''(w)zv).
\]
Write $(w_\tau,s_\tau)$ for the corresponding primal and $s$-components
of $\Gamma(\tau)$, set $e_s:=q-q_h$, and put
\[
   \delta_\tau:=\rho'(u)-\rho'(w_\tau).
\]
Since $q=\rho(u)$ and
$w_\tau-u=-(1-\tau)e$, $s_\tau-q=-(1-\tau)e_s$, boundedness of
$\rho'$ and $\rho''$ gives
\[
   \|\rho(w_\tau)-s_\tau\|
   \lesssim E,
   \qquad
   |\delta_\tau|\lesssim |e|.
\]
The difference of the first terms in the two Hessians can be written as
\begin{align*}
 &(\rho'(u)e-e_s,\rho'(u)\phi_h-(\phi_h)_s)
 -(\rho'(w_\tau)e-e_s,\rho'(w_\tau)\phi_h-(\phi_h)_s)
 \\
 &\quad =
 (\delta_\tau e,\rho'(u)\phi_h-(\phi_h)_s)
 +(\rho'(w_\tau)e-e_s,\delta_\tau\phi_h).
\end{align*}
For the first term, $H^1(\Omega)\hookrightarrow L_4(\Omega)$ for $d\le3$
and \eqref{eq:dual-stability} yield
\[
 \|\delta_\tau e\|
 \lesssim \|e\|_{L_4}^2\lesssim E^2,
 \qquad
 \|\rho'(u)\phi_h-(\phi_h)_s\|
 \lesssim \|\bphi_h\|_{U}\lesssim \|e\|.
\]
For the second term,
\[
 \|\rho'(w_\tau)e-e_s\|\lesssim E,
 \qquad
 \|\delta_\tau\phi_h\|
 \lesssim \|e\|\,\|\phi_h\|_{L_\infty}
 \lesssim E\|e\|.
\]
Here and below Poincar\'e's inequality is used to bound the $L_2$-norm of
the primal error by $E$.  Finally, the residual contribution satisfies
\[
 \bigl|(\rho(w_\tau)-s_\tau,\rho''(w_\tau)e\phi_h)\bigr|
 \lesssim E\,\|e\|\,\|\phi_h\|_{L_\infty}
 \lesssim E^2\|e\|.
\]
Thus the complete $\rho$-Hessian difference is bounded by
$CE^2\|e\|$.  The same argument applies to the $\gamma$-term.  Hence
\[
  \|e\|^2\lesssim h^sE\|e\|+E^2\|e\|,
\]
which proves \eqref{eq:nonlinear-AN}.

The proof of Theorem~\ref{thm_Cea} shows that
\[
  \|\bu-\bu_h\|_{U}
  \lesssim
  \|\bu-\bw_h\|_{U}+\|u-u_h\|
  \quad\forall\bw_h\in U_h.
\]
Combining with \eqref{eq:nonlinear-AN} gives
\[
 E\lesssim
 \|\bu-\bw_h\|_{U}+h^sE+E^2.
\]
For sufficiently small $h$ and $E$, the last two terms are absorbed into the
left-hand side, proving \eqref{eq:local-cea}.
\end{proof}

\section{Numerics} \label{sec_num}

\subsection{Nonlinear systems} \label{sec_Newton}
We consider differentiable nonlinear functions $\rho$ and $\gamma$.
The Euler--Lagrange equations for the minimum of
$\new{\frac 12}\|B\bw_h-L\|_{V_h^*}^{\new{2}}$ with solution $\bu_h\in U_h$ are
\[
   b(\bu_h,\Theta_h \bw_h)=L(\Theta_h\bw_h)\quad\forall \bw_h\in U_h
\]
with discrete trial-to-test operator $\Theta_h:\;U_h\to V_h$ defined by
\[
   \ip{\Theta_h\bw_h}{v_h}_V = b(\bw_h,v_h)\quad\forall v_h\in V_h,
\]
cf.~\eqref{prob_bLTheta}.
As discussed previously, after \eqref{min}, this is a DPG scheme.
Again, we refer to \cite{DemkowiczG_25_DPG}.

Adding the variations of the nonlinear $L_2$-terms,
we obtain the Euler--Lagrange equations of minimum residual discretization \eqref{min_disc},
\begin{align} \label{EL}
   &\
   b(\bu_h,\Theta_h \bw_h)
   + \vdual{\rho(u_h)-q_h}{\rho'(u_h)w_h-s_h}
   + \vdual{\gamma(u_h)-r_h}{\gamma'(u_h)w_h-t_h}
   = L(\Theta_h\bw_h)
\end{align}
for any $\bw_h=(w_h,\htau_h,s_h,t_h)\in U_h$,
and with solution $\bu_h=(u_h,\hsigma_h,q_h,r_h)\in U_h$.
For twice differentiable functions $\rho$ and $\gamma$,
we solve these systems by the Newton method.
We stop the iterations when the global residual
\begin{equation} \label{Res}
   \mathrm{Res}
   := \Bigl(\|B\bu_h-L\|_{V_h^*}^2 + \|\rho(u_h)-q_h\|^2 + \|\gamma(u_h)-r_h\|^2 \Bigr)^{1/2}
\end{equation}
is below a certain threshold.

\subsection{Adaptivity} \label{sec_adap}
One of the main advantages of DPG and least-squares methods is that error estimation
and mesh refinement indicators are intrinsic.
For details on the DPG a posteriori error analysis, we refer to
\cite{CarstensenDG_14_PEC} and \cite[Sec.~6]{DemkowiczG_25_DPG}.
By the locality of $L_2$-norms and the product test space $V_h$,
the global residual Res \eqref{Res} consists of local terms assigned to elements,
in particular
\begin{equation*} 
   \|B\bu_h-L\|_{V_h^*}^2
   = \sum_{\el\in\mesh} \|\new{B\bu_h-L}\|_{V_h(\el)^*}^2.
\end{equation*}
Therefore,
\(
   \mathrm{Res}^2 = \sum_{\el\in\mesh} \mathrm{Res}(\el)^2
\)
with
\[
   \mathrm{Res}(\el)^2
   := \|\new{B\bu_h-L}\|_{V_h(\el)^*}^2 + \|\rho(u_h)-q_h\|_\el^2 + \|\gamma(u_h)-r_h\|_\el^2,
\]
and $\mathrm{Res}$ is computable. The indicators $\mathrm{Res}(\el)$
can be used to steer adaptive mesh refinement.
We use newest vertex bisection \cite{Mitchel_17_30Y} and
the bulk criterion \cite{Doerfler_96_CAA} with parameter $\theta\in (0,1)$
to mark the smallest set of elements $\wilde\mesh\subset\mesh$ so that
\[
   \sum_{\el\in\wilde\mesh} \mathrm{Res}(\el)^2
   \ge
   \theta \mathrm{Res}^2.
\]

\subsection{Experiments} \label{sec_exp}
We present numerical experiments in two space dimensions,
and compare in some cases our method
(referred to as ``DPG-LS/Newton'') with the previously mentioned
Gauss--Newton DPG (``Gauss--Newton/DPG'') method
initially proposed by Demkowicz and co-workers \cite{ChanDM_14_DMS,RobertsDM_15_DPG}.
We use the discrete spaces $U_h$, $V_h$ defined in \eqref{Uh}, \eqref{Vh}
and select the lowest order $p=0$. We start with a uniform mesh of triangles and generate
sequences of meshes with $N$ elements,
based on uniform refinements or adaptive refinements with $\theta=0.5$
as described in \S\ref{sec_adap}.
(For comparison we test Gauss--Newton/DPG also with $\theta=0.4$ and $\theta=0.6$.)
The existence of a Fortin operator is guaranteed by Corollary~\ref{cor_Fortin}.
We consider examples with manufactured solutions by prescribing $u$, $\rho$, and $\gamma$,
and setting $\kappa=I$, $\beta=(1,2)^\top$.
The right-hand side functions $f$ are calculated accordingly.
The nonlinear forms are integrated by the $3$-point Gauss formula on elements.
For the residual and error calculation, we switch to $7$ points.
\new{To assess the influence of numerical quadrature, we repeated Example 2
using a $7$-point rule for assembling the nonlinear terms and a $16$-point rule
for the residual and error evaluation. The results on uniformly refined meshes were
virtually unchanged, while adaptive refinement produced slightly different meshes
but the same convergence behavior and Newton iteration counts.}

For every mesh, the Newton iterations for the solution of \eqref{EL} start with trivial
approximations of $\hsigma_h$, $q_h$ and $r_h$,
and the nodal interpolation of Dirichlet data and trivial extension to an
element of $\cP^1(\mesh)\cap H^1(\Omega)$ as initial guess for $u_h$
(for Example~3 with square-root nonlinearities, we extend the initial guess
for $u_h$ to $1$ at interior nodes).
Of course, this approach can be adapted to nested iterations where solutions
on a given mesh are used to generate an initial guess on a refined mesh.
The Newton iterations are stopped when $\mathrm{Res}<10^{-6}$, cf.~\eqref{Res}.
Throughout, the observed iteration numbers do not exceed $4$ for DPG-LS/Newton
and $5$ for Gauss--Newton/DPG. Below, we report on the residuals
\begin{equation} \label{res_terms}
   \|\mathrm{res}\|_{V_h^*} := \|B\bu_h-L\|_{V_h^*},\quad
   \|\rho(u_h)-q_h\|,\quad \|\gamma(u_h)-r_h\|,
\end{equation}
the individual terms of the error
\begin{equation} \label{err_terms}
   |\bu-\bu_h|_U
   := \Bigl(\|\grad(u-u_h)\|^2 + \|\rho(u)-q_h\|^2 + \|\gamma(u)-r_h\|^2\Bigr)^{1/2},\quad
   \text{and}\quad \|u-u_h\|.
\end{equation}
(Note that $|\bu-\bu_h|_U$ is lacking the $\hsigma$-component of $\|\bu-\bu_h\|_U$.)

Next we present three examples.
\new{Example~1 has homogeneous Dirichlet data and is covered by
our Aubin--Nitsche analysis in Section~\ref{sec_AN}.
The data and discrete spaces satisfy \ref{ass_B1}--\ref{ass_B4}, so
Theorem~\ref{thm_AN} applies to a corresponding minimizer.
The second and third examples are not covered by our analysis.}

\bigskip \noindent
{\bf Example 1.}
We consider the unit square $\Omega=(0,1)^2$ and select
$\rho(u)=\cos(u)$, $\gamma(u)=\arctan(u)$, and $u(x,y)=\sin(2\pi x)\sin(\pi y)$.
For uniform mesh refinements, Figure~\ref{fig2_err_unif} shows the residual
and error terms listed in \eqref{res_terms} and \eqref{err_terms}, respectively.
They confirm the expected convergence order $O(N^{-1/2})$ (linear with respect to the mesh size)
of both the global residual $\mathrm{Res}$ and the error $|\bu-\bu_h|_U$. 
We also observe that $\|u-u_h\|=O(N^{-1})$ is of higher order than
$\|\grad (u-u_h)\|=O(N^{-1/2})$,
\new{consistent with the closeness assumptions of Theorem~\ref{thm_Cea}.
In this case, the assumptions \ref{ass_B1}--\ref{ass_B3} can be verified
and \ref{ass_B4} holds by construction and selection of $V_h$.}

For this example we compare the performance \new{of} the DPG-LS/Newton scheme
with that of the Gauss--Newton/DPG approach where
we use the analogous primal DPG setting without additional variables $q, r$,
and start the iterations with the same initial guess.
We stop the Gauss--Newton iteration when the residual norm
\new{$\|B_{\mathrm{GN}}\bw_h\|_{V_h^*}$} of the correction $\bw_h$
\new{with linearized operator $B_\mathrm{GN}$}
is below the threshold $10^{-6}$ from before.
Again, we consider a sequence of uniformly refined meshes.

Both methods are implemented in Octave \cite{Octave},
with the same structure and basic routines including numerical integration.
In both cases, the Newton iterations are optimized by
pre-calculating all iteration-independent terms, and linear systems are
solved by the backslash command.

Table~\ref{tab2_cpu} shows the CPU times in seconds on a
Linux Desktop PC with Intel processor i7-12700T and 32GB RAM. We also report on the
iteration numbers and the achieved errors $\|\grad(u-u_h)\|$.
It appears that, for the given range of numbers of elements, the CPU times
are a bit smaller for the
\new{Gauss--Newton/DPG method (with initially more iterations),
though the difference becomes small for the largest systems, see~Figure~\ref{fig2_cpu}.
To put the results into perspective,
we added curves that decay like time$^{-1/4}$ and time$^{-1/3}$.
If we denote by $\text{time}\sim\mathrm{\# dof}^\alpha$
the observed dependence of CPU time on the number of degrees of freedom dof,
the two reference curves correspond to $\alpha=2$ and $\alpha=3/2$, respectively,
and bracket the observed scaling
(note that $\|\grad(u-u_h)\|\sim N^{-1/2}\sim\mathrm{\# dof}^{-1/2}$).
We remark that our selection of optimized local test spaces of dimension $4$
is advantageous for the Gauss--Newton scheme because the repeated calculation of
the trial-to-test operation is the expensive part. Numerical tests with larger
test spaces (not necessary for this simple model problem) show almost identical
CPU times for both methods.
}

\bigskip \noindent
{\bf Example 2.}
We consider the $L$-shaped domain $(-1,1)^2\setminus[0,1)\times (-1,0]$, select
$\rho(u)=u^2$, $\gamma(u)=u^3$, and prescribe the typical Laplace singularity
$u(r,\phi)= r^{2/3} \cos(2/3\phi)$ for polar coordinates $(r,\phi)$
centered at the origin and with \new{$\phi=0$ along the ray with polar angle} $3\pi/4$.
Note that $u$ satisfies  $\Delta u=0$ in $\Omega$ and vanishes at the edges
that meet at the origin.
The harmonicity of this function is irrelevant for our problem
(though it simplifies the calculation of $f$).
\new{It nevertheless} provides a typical test case of reduced regularity to check for convergence
of the scheme. Observe that the nonlinear functions are
(uniformly) Lipschitz continuous only on bounded intervals, i.e., in a neighborhood
of the range of the bounded solution $u$ on $\Omega$, cf.~Remark~\ref{rem_Lip}.

We plot the error and residual terms as in the first example.
Figure~\ref{fig4_err_unif} shows the results for uniformly refined meshes.
As expected, convergence orders are reduced to $O(N^{-1/3})$ which corresponds
to the regularity $u\in H^{2/3+1-\epsilon}(\Omega)$ for any $\epsilon>0$.
We also observe an increased convergence order $\|u-u_h\|=O(N^{-2/3})$,
which suggests that the Aubin--Nitsche assumption of Theorem~\ref{thm_Cea} is satisfied.
The results for adaptive mesh refinements from Figure~\ref{fig4_err_adap}
indicate re-established optimal convergence orders of $O(N^{-1/2})$
and $\|u-u_h\|=O(N^{-1})$, again \new{providing evidence for the}
validity of the Aubin--Nitsche assumption in Theorem~\ref{thm_Cea}.
For illustration, Figure~\ref{fig4_sol} presents the exact solutions $u$, $q=\rho(u)$,
$r=\gamma(u)$ in the top row and, in the bottom row,
their approximations after $4$ adaptive refinements on a mesh of $N=191$ elements.
As expected, the mesh is refined towards the re-entrant corner (barely visible).

\bigskip \noindent
{\bf Example 3.}
In this example, we compare the performance of the DPG-LS/Newton and Gauss--Newton/DPG
schemes concerning adaptivity. In the case of Example 2,
there are no relevant differences (results not reported).
However, as discussed in the introduction, adaptivity steered by the inherent residuals
is more natural for the DPG-LS/Newton variant
as it controls the error of the nonlinear problem.
The residual of the Gauss--Newton/DPG method only controls the error of the linearized problem.
We underline this claim with the manufactured solution
\new{$u(x,y)=1/10+x^2+y^2$} on the square $\Omega=(0,1)^2$
and nonlinear functions $\rho(u)=\gamma(u)=\sqrt{u}$.
Since $u\ge 1/10$, the functions $\rho=\gamma=\sqrt{\cdot}$ are Lipschitz continuous
in a neighborhood of the range of the exact solution, cf.~Remark~\ref{rem_Lip}.
To account for this condition we start the Newton iterations by setting $u_h$ to $1$
at interior nodes and, as before, to values of the exact solution at the boundary nodes.
Throughout, the stopping criterion (residual $<10^{-6}$) is reached after $3$ iterations.
For the adaptive algorithm described in Section~\ref{sec_adap}, as performed for Example 2,
we obtain the error dependence on the CPU time as illustrated by the left part of
Figure~\ref{fig5_err}. It suggests superiority of the former method for this case.
Since residual terms are different for both methods, we tested the Gauss--Newton/DPG
variant also with refinement parameters $\theta=0.4$ and $\theta=0.6$,
not leading to noticeable differences.
\new{In the right panel} of Figure~\ref{fig5_err} we plot the errors versus the number of elements.
They indicate that the difference in efficiency is due \new{to} different mesh refinements
rather than different efficiencies of the methods.
The square root behavior of $\rho$ and $\gamma$, even though smooth for positive arguments,
seems to cause a non-negligible effect on the linearization of the problem and, thus,
on the residual for error control and adaptivity.

\begin{figure}
\begin{center}
\includegraphics[width=0.8\textwidth]{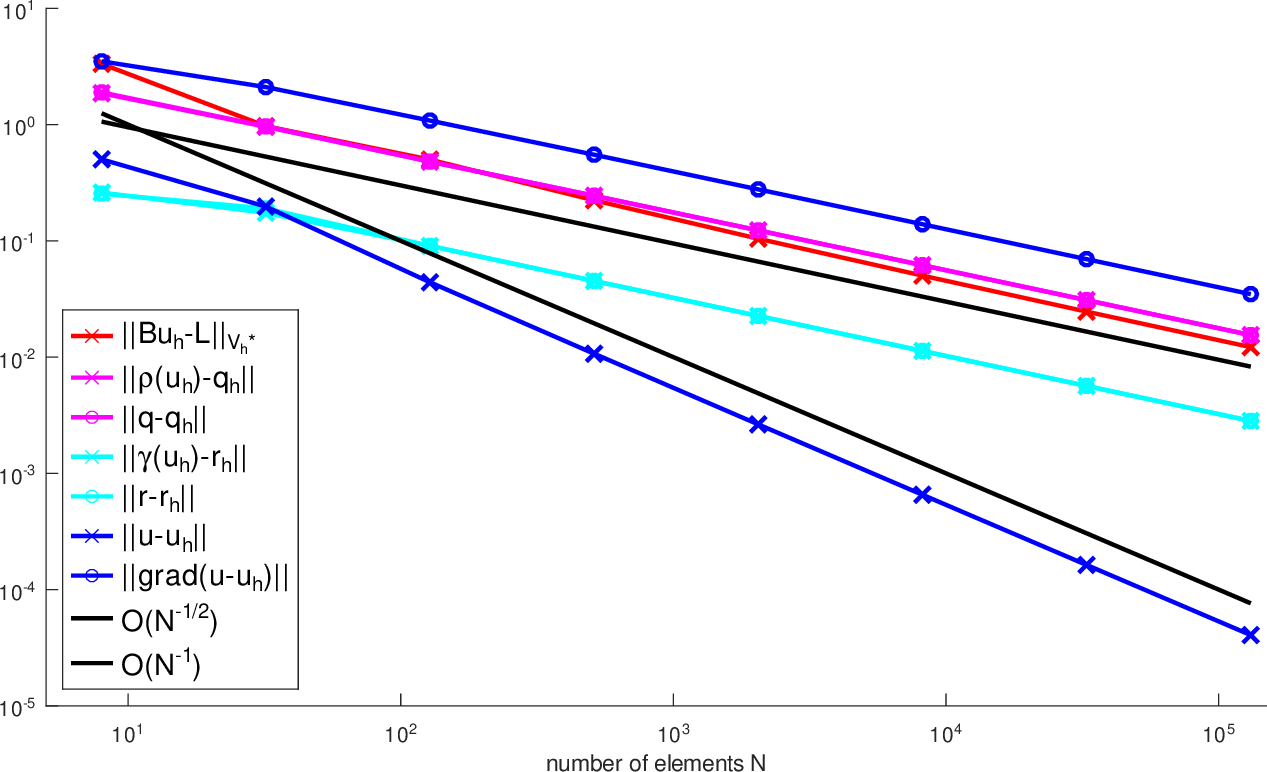}
\end{center}
\caption{Example 1. Errors and residuals, uniform refinements.}
\label{fig2_err_unif}
\end{figure}

\begin{table}
\begin{center}
\begin{tabular}{r|rcc|rcc}
    & \multicolumn{3}{c|}{DPG-LS/Newton} & \multicolumn{3}{c}{Gauss--Newton/DPG}\\
$N$ & CPU [s] & \#iter & $\|\grad(u-u_h)\|$ & CPU [s] & \#iter & $\|\grad(u-u_h)\|$\\\hline
    32 &   0.8 & 3 & 2.1030 &    0.9 & 5 & 1.8153 \\
   128 &   1.1 & 3 & 1.0829 &    1.2 & 4 & 0.9438 \\
   512 &   2.2 & 3 & 0.5509 &    2.1 & 3 & 0.4765 \\
  2048 &  27.3 & 3 & 0.2772 &   18.7 & 3 & 0.2388 \\
  8192 &   149 & 3 & 0.1390 &   64.6 & 3 & 0.1195 \\
 32768 &  1100 & 3 & 0.0696 &  838.5 & 3 & 0.0598 \\
131072 & 16663 & 3 & 0.0348 &  15212 & 3 & 0.0299
\end{tabular}
\end{center}
\caption{Example 1. Comparing CPU times for the DPG-least squares/Newton
                    and Gauss--Newton/DPG schemes.}
\label{tab2_cpu}
\end{table}

\begin{figure}
\begin{center}
\includegraphics[width=\textwidth]{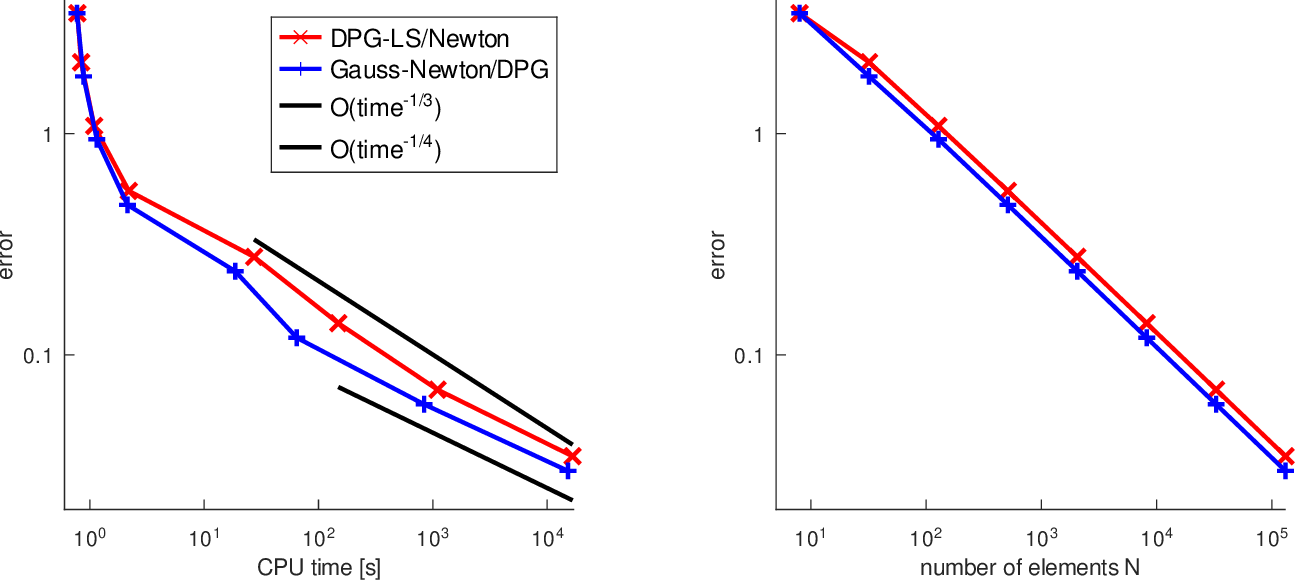}
\end{center}
\caption{Example 1. The error $\|\grad(u-u_h)\|$ versus CPU time (left)
         and $N$ (right) for the DPG-least squares/Newton and Gauss--Newton/DPG schemes.}
\label{fig2_cpu}
\end{figure}

\begin{figure}
\begin{center}
\includegraphics[width=0.8\textwidth]{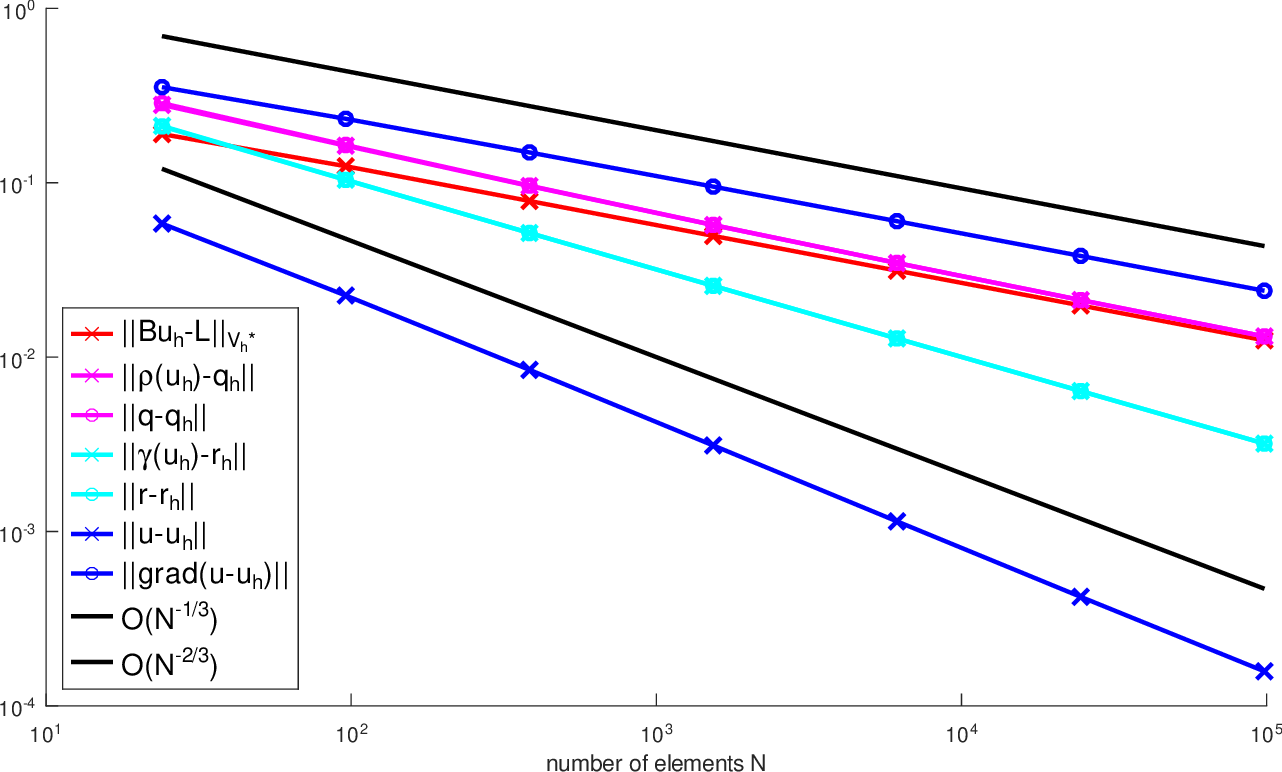}
\end{center}
\caption{Example 2. Errors and residuals, uniform refinements.}
\label{fig4_err_unif}
\end{figure}

\begin{figure}
\begin{center}
\includegraphics[width=0.8\textwidth]{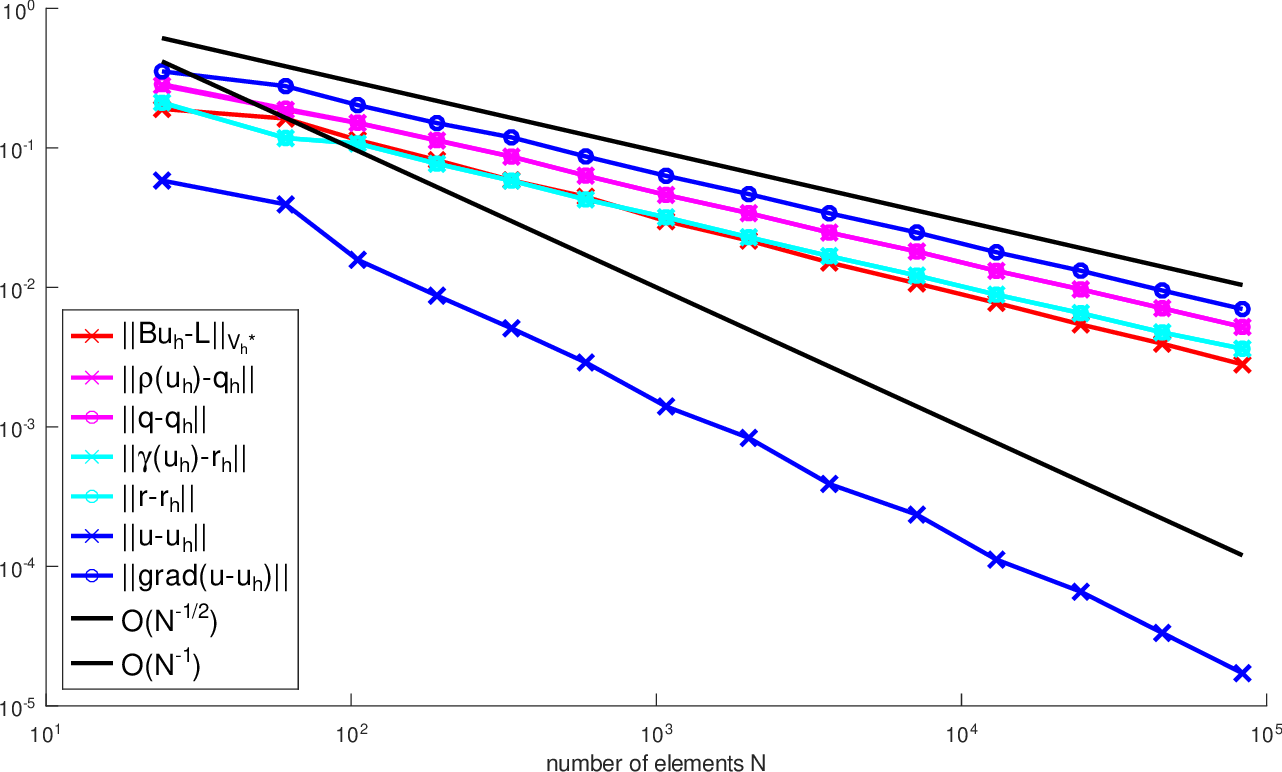}
\end{center}
\caption{Example 2. Errors and residuals, adaptive refinements.}
\label{fig4_err_adap}
\end{figure}

\begin{figure}
\begin{center}
\includegraphics[height=0.5\textwidth,width=\textwidth]{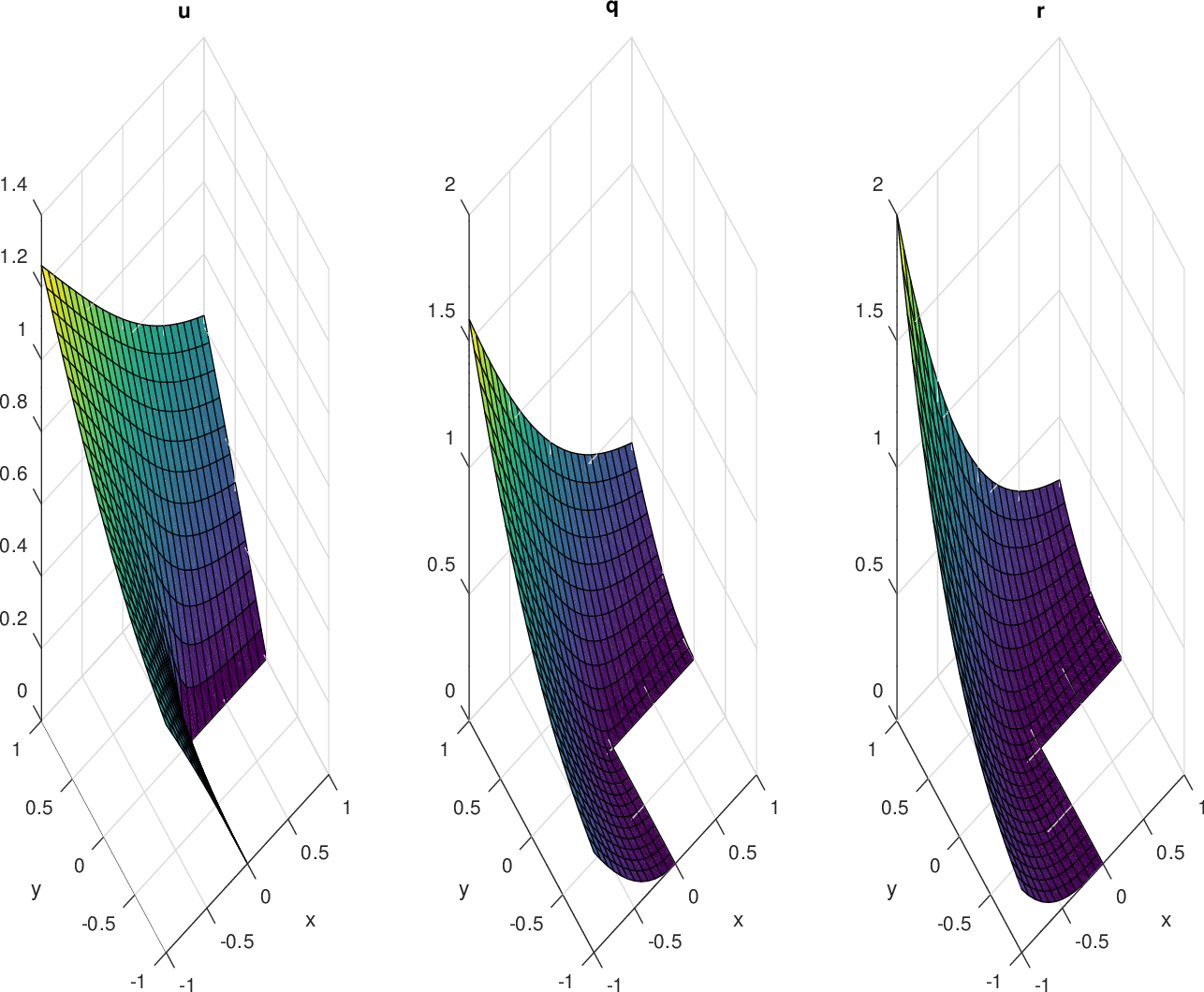}\\[-0.5em]
\includegraphics[height=0.5\textwidth,width=\textwidth]{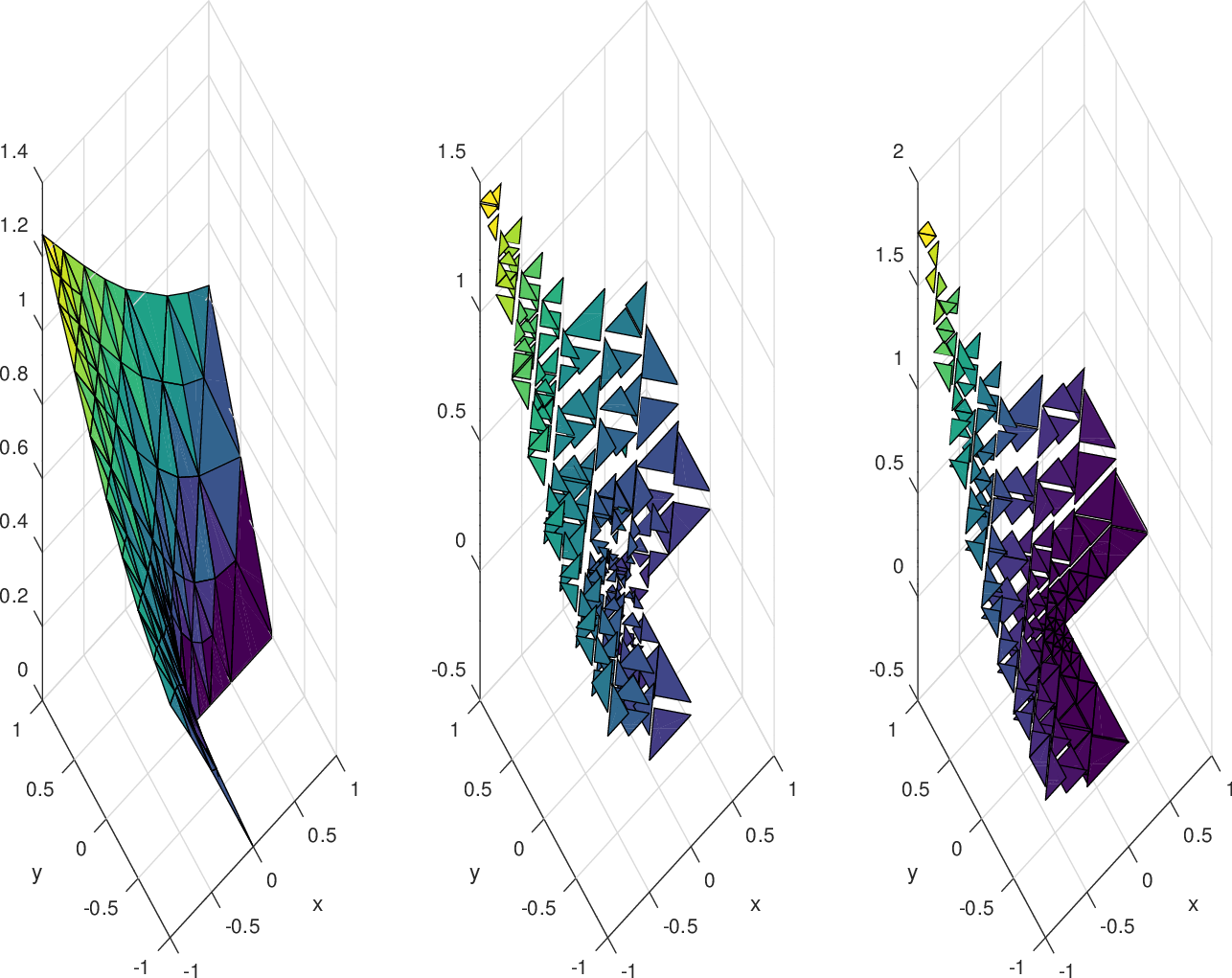}
\end{center}
\caption{Example 2. Exact solutions $u$, $q=\rho(u)$, $r=\gamma(u)$ (top)
and their approximations (bottom) after $4$ adaptive refinements, $N=191$.}
\label{fig4_sol}
\end{figure}

\begin{figure}
\begin{center}
\includegraphics[width=\textwidth]{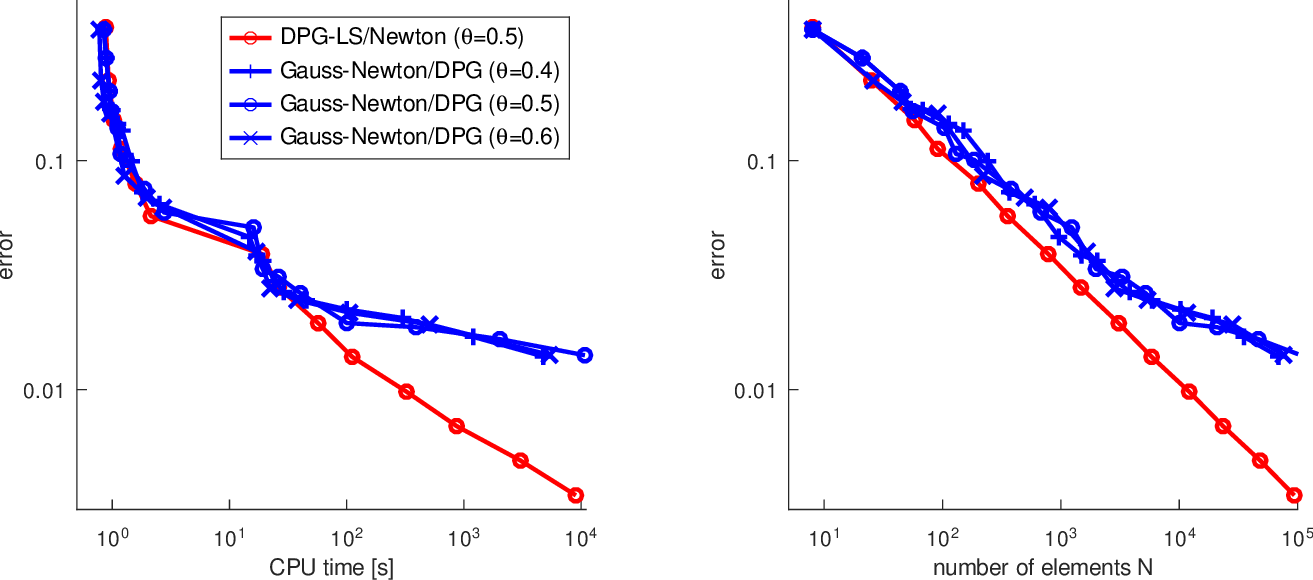}
\end{center}
\caption{Example 3. The error $\|\grad(u-u_h)\|$ versus CPU time (left)
         and $N$ (right) for the DPG-LS/Newton and Gauss--Newton/DPG schemes.}
\label{fig5_err}
\end{figure}

\FloatBarrier
\section*{Conclusions}

We have proposed a simple way of extending the DPG method to certain
semilinear problems. Main advantages of the method are the separation of the
nonlinearities from the generally more elaborate DPG framework,
the presence of a built-in a posteriori error estimator comprised of local terms,
and the fact that the estimator controls the full nonlinear problem.
\new{We have shown that discrete minimizers sufficiently close to an exact solution
satisfy a quasi-optimal error estimate. Under the additional regularity and
stability assumptions of Section~\ref{sec_AN},
the required nonlinear Aubin--Nitsche estimate is established.}
Numerical experiments illustrate the expected convergence properties
and suggest that the computational efficiency is similar to \new{that} of the
Gauss--Newton scheme \new{based on the corresponding primal DPG discretization.}
For certain nonlinear terms, the nonlinear residual of the new method leads to
better mesh refinements than the residual of the linearized problem from the
Gauss--Newton approach.

This work was motivated by the parabolic Richards equation. An extension
of the combined DPG-least squares scheme to this and other parabolic problems
is left to future research.

\appendix
\newpage
\section{Proof of Theorem~\ref{thm_ex}} \label{sec_app}

Let $A$ denote the differential operator from \eqref{pde}, i.e.,
\[
   \dual{A(u)}{\du} := \vdual{\kappa\grad u+\rho(u)\beta}{\grad\du} + \vdual{\gamma(u)}{\du}
   \quad \text{for}\ u,\du\in H^1_0(\Omega).
\]
The surjectivity of $A:\;H_0^1(\Omega)\to H^1_0(\Omega)^*$ follows from
Br\'ezis' theorem \cite[Theorem~2.6]{Roubicek_05_NPD} since $A$ is
bounded, pseudo-monotone and coercive, as we verify next.

The boundedness of $A$ is immediate by assumptions (A1), (A2):
\[
   \|A(u)\|_{H^1_0(\Omega)^*}
   \le
   \bigl(\|\kappa\|_{\infty,\Omega} + L_\rho|\beta| + L_\gamma\bigr) \|u\|_{1,\Omega}
   + \bigl(|\rho(0)|\,|\beta| + |\gamma(0)|\bigr)\, |\Omega|^{1/2}
\]
with $\|\kappa\|_{\infty,\Omega}:=\mathrm{ess\, sup}_{x\in\Omega}|\kappa(x)|_2$ for
the induced Euclidean matrix norm $|\cdot|_2$.
For ease of reference let us introduce the notation from \cite{Roubicek_05_NPD},
\[
   a(x,r,s) := \kappa(x) s + \rho(r)\beta,\quad
   c(x,r,s) := \gamma(r)\quad \text{for}\ x\in\Omega,\ r\in\R,\ s\in\R^d.
\]
Both $a$ and $c$ are Carath\'eodory functions (measurable in $x$ and continuous in $r,s$).
The pseudo-monotonicity of $A$ follows from \cite[Lemma~2.32]{Roubicek_05_NPD} by
verifying with (A1), (A2) that
\begin{align*}
   &|a(x,r,s)| \le
     \|\kappa\|_{\infty,\Omega} |s| + L_\rho |\beta|\,|r| + |\rho(0)|\,|\beta|
   &&\text{\cite[relation~(2.55a)]{Roubicek_05_NPD}},\\
   &|c(x,r,s)| \le L_\gamma|r| + |\gamma(0)|
   &&\text{\cite[relation~(2.55c)]{Roubicek_05_NPD}},
\end{align*}
the stronger form of relation \cite[(2.65)]{Roubicek_05_NPD}
\begin{align*}
   &(a(x,r,s)-a(x,r,\tilde s))\cdot(s-\tilde s) \ge \alpha |s-\tilde s|^2
      \quad\forall r\in\R,\ \forall s,\tilde s\in\R^d,\ \text{for a.e. $x\in\Omega$},
\end{align*}
and using the fact that function $c$ is independent of $s$.

We employ \cite[Lemma~2.35]{Roubicek_05_NPD} to check the coercivity of $A$.
Note that Assumption (A3) implies $\gamma(r)r\ge\gamma(0)r+l_\gamma r^2$.
Therefore, also making use of (A1), (A2),
\begin{align} \label{a_ell}
   &a(x,r,s)\cdot s+c(x,r,s)r
   = (\kappa(x)s + \rho(r)\beta)\cdot s+\gamma(r)r
   \ge \alpha|s|^2 - \|\rho\|_{\infty,\R}|\beta|\,|s| + \gamma(0)r + l_\gamma r^2
\end{align}
for $r\in\R$, $s\in\R^d$ and almost every $x\in\Omega$.
An application of Young's inequality shows that the lower bound in \eqref{a_ell}
is bounded below quadratically in $|s|$ and $r$, minus a constant,
implying condition (2.92a) in \cite[Lemma~2.35]{Roubicek_05_NPD},
and thus the coercivity of $A:H^1_0(\Omega)\to H^1_0(\Omega)^*$.
By \cite[Theorem~2.6]{Roubicek_05_NPD} this finishes the proof of existence of a solution.

For uniqueness, it is immediate to verify that $A:\;H^1_0(\Omega)\to H^1_0(\Omega)^*$ is Lipschitz
continuous. Furthermore, by standard estimates as in \eqref{a_ell} one sees with
(A1), (A2), (A3) that
\[
   \dual{A(u_1)-A(u_2)}{u_1-u_2}
   \ge \min\{\alpha-L_\rho|\beta|\frac \epsilon 2,
             l_\gamma-L_\rho|\beta|\frac 1{2\epsilon}\}\|u_1-u_2\|_{1,\Omega}^2
   \ \forall u_1,u_2\in H^1_0(\Omega),\; \epsilon>0.
\]
If (A4) holds, that is,
\(
   L_\rho^2 |\beta|^2<4 \alpha l_\gamma,
\)
then there exists an $\epsilon>0$ so that the minimum above is positive.
Therefore, in case (A1)--(A4) are satisfied, $A$ is Lipschitz continuous and
strongly monotone, and \eqref{pde} is uniquely solvable.


\bigskip \noindent
{\bf Acknowledgment.}
In~Section~\ref{sec_AN}, we used ChatGPT (OpenAI) to explore proof strategies
and assist with the exposition of the Aubin--Nitsche argument.
We take full responsibility for the published material.

\bibliographystyle{elsarticle-num}
\bibliography{paper}
\end{document}

%% file: header.tex
\newtheorem{theorem}{Theorem}
\newtheorem{lemma}[theorem]{Lemma}
\newtheorem{cor}[theorem]{Corollary}
\newtheorem{prop}[theorem]{Proposition}
\newtheorem{remark}[theorem]{Remark}
\theoremstyle{definition}

\newcommand{\wilde}{\widetilde}
\newcommand{\R}{\ensuremath{\mathbb{R}}}
\newcommand{\mesh}{\mathcal{T}}
\newcommand{\el}{T}

\newcommand{\face}{F}
\newcommand{\cS}{\ensuremath{\mathcal{S}}}
\newcommand{\cN}{\ensuremath{\mathcal{N}}}
\newcommand{\cP}{\ensuremath{\mathcal{P}}}
\newcommand{\cF}{\ensuremath{\mathcal{F}}}
\newcommand{\cL}{\ensuremath{\mathcal{L}}}
\newcommand{\cH}{\ensuremath{\mathcal{H}}}

\newcommand{\bu}{\boldsymbol{u}}
\newcommand{\bw}{\boldsymbol{w}}
\newcommand{\hsigma}{{\widehat\sigma}}
\newcommand{\htau}{{\widehat\tau}}
\newcommand{\bphi}{\boldsymbol{\phi}}
\newcommand{\bz}{\boldsymbol{z}}

\newcommand{\du}{\delta\!u}

\DeclareMathOperator{\trn}{tr_\mathit{n}}
\DeclareMathOperator*{\argmin}{\textrm{arg\,min}}
\newcommand{\grad}{\nabla}
\let\div\undefined
\DeclareMathOperator{\div}{div}
\newcommand{\opF}{{\Pi_F}}

\newcommand{\cE}{\mathcal{E}}
\newcommand{\CF}{C_F}

\newcommand{\dual}[2]{\langle#1\hspace*{.5mm},#2\rangle}
\newcommand{\vdual}[2]{(#1\hspace*{.5mm},#2)}
\newcommand{\ip}[2]{\llangle#1\hspace*{.5mm},#2\rrangle}